\documentclass[preprint,11pt]{elsarticle_2}
\usepackage{stmaryrd}

\usepackage{pdfpages}

\usepackage{amscd}
\usepackage{amsmath,amsthm,amssymb}

\usepackage{mathptmx}

\usepackage{amsfonts}

\usepackage{float}

\usepackage{xcolor}
\usepackage{colortbl}

\definecolor{WhiteRed}{rgb}{1.00,0.50,0.50}
\definecolor{RedOrange}{rgb}{1.00,0.65,0.47}
\definecolor{whiteblue}{rgb}{0.58,0.58,1.00}

\usepackage{multirow}

\DeclareFontFamily{U}  {MnSymbolC}{}
  
  \DeclareFontShape{U}{MnSymbolC}{m}{n}{
    <-6>  MnSymbolC5
   <6-7>  MnSymbolC6
   <7-8>  MnSymbolC7
   <8-9>  MnSymbolC8
   <9-10> MnSymbolC9
  <10-12> MnSymbolC10
  <12->   MnSymbolC12}{}
\DeclareFontShape{U}{MnSymbolC}{b}{n}{
    <-6>  MnSymbolC-Bold5
   <6-7>  MnSymbolC-Bold6
   <7-8>  MnSymbolC-Bold7
   <8-9>  MnSymbolC-Bold8
   <9-10> MnSymbolC-Bold9
  <10-12> MnSymbolC-Bold10
  <12->   MnSymbolC-Bold12}{}

\DeclareSymbolFont{MnSyC} {U} {MnSymbolC}{m}{n}

\usepackage{rotating}

\DeclareMathSymbol{\veedot}{\mathop}{MnSyC}{47}

\usepackage{relsize}

\usepackage{color}
\usepackage{xspace}
\usepackage{cleveref}

\usepackage{bussproofs}

\usepackage{tikz}
\usepackage{bpextra}

\usepackage{enumerate}

\usepackage{array}
\newcolumntype{C}[1]{>{\centering\arraybackslash}p{#1}}

\DeclareSymbolFont{letters}{OML}{cmm}{m}{it}

\DeclareMathAlphabet{\mathcal}{OMS}{cmsy}{m}{n}

\newtheorem{theorem}{Theorem}[section]
\newtheorem{corollary}[theorem]{Corollary}
\newtheorem{lemma}[theorem]{Lemma}
\newtheorem{proposition}[theorem]{Proposition}

\newtheorem{example}[theorem]{Example}

\newtheorem{fact}[theorem]{Fact}

\newtheorem{remark}[theorem]{\textrm{\textbf{Remark}}}

\theoremstyle{definition}
\newtheorem{definition}[theorem]{Definition}

\newtheorem*{claim*}{Claim}

\newcommand{\CPL}{\ensuremath{\mathbf{CPL}}\xspace}

\newcommand{\PML}{\ensuremath{\mathbf{CPL}(\might)}\xspace}

\newcommand{\PU}{\ensuremath{\mathbf{CPL}(\ror)}\xspace}
\newcommand{\PUw}{\ensuremath{\mathbf{CPL(\ror)}}\xspace}

\newcommand{\PAm}{\ensuremath{\mathbf{CPL}(\Upsilon)}\xspace}

\newcommand{\Picnst}{\ensuremath{\CPL(\idep(\cdot))}\xspace}

\newcommand{\PInc}{\ensuremath{\mathop{\mathbf{CPL}(\subseteq)}}\xspace}
\newcommand{\PIncz}{\ensuremath{\mathop{\mathbf{CPL}(\subseteq_0)}}\xspace}

\newcommand{\LL}{\ensuremath{\mathsf{L}}\xspace}

\newcommand{\dblsetminus}{\mathbin{{\setminus}\mspace{-5mu}{\setminus}}}

\newcommand{\dep}{\ensuremath{\mathop{=\!}}\xspace}
\newcommand{\idep}{\ensuremath{\mathop{\neq\!}}\xspace}

\newcommand{\sor}{\ensuremath{\vee}\xspace}
\newcommand{\ior}{\ensuremath{\mathbin{\rotatebox[origin=c]{-90}{$\geqslant$}}}\xspace}

\DeclareMathOperator*{\bigsor}{\bigvee}

\newcommand{\bigveedot}{\ensuremath{\mathop{\cdot{\hspace{-1.4ex}\bigvee}}}\xspace}

\newcommand{\ror}{\ensuremath{\veedot}\xspace}
\newcommand{\bigror}{\xspace\ensuremath{\bigveedot}\xspace}
\newcommand{\bigroril}{\ensuremath{\xspace\hspace{.15ex}\mathop{\cdot\hspace{-1.4ex}\bigvee}}\xspace}

\newcommand{\might}{\ensuremath{\triangledown}\xspace}

\newcommand{\vvee}{\raisebox{1pt}{\ensuremath{\,\mathop{\mathsmaller{\mathsmaller{\dblsetminus\hspace{-0.23ex}/}}}\,}}}

\newcommand{\anm}{\ensuremath{\mathop{\neq\!}}\xspace}

\newcommand{\ci}{\ensuremath{\wedge\textsf{I}}\xspace}

\newcommand{\ce}{\ensuremath{\wedge\textsf{E}}\xspace}
\newcommand{\sori}{\ensuremath{\sor\textsf{I}}\xspace}
\newcommand{\sore}{\ensuremath{\sor\textsf{E}}\xspace}

\newcommand{\dstr}{\ensuremath{\textsf{Dstr}}\xspace}

\newcommand{\topi}{\ensuremath{\top\textsf{I}}\xspace}

\newcommand{\incext}{\ensuremath{\subseteq\!\!\textsf{Ext}}\xspace}
\newcommand{\incexc}{\ensuremath{\subseteq\!\mathsf{Exc}}\xspace}
\newcommand{\incctr}{\ensuremath{\subseteq\!\mathsf{Ctr}}\xspace}
\newcommand{\inctrs}{\ensuremath{\subseteq\!\!\mathsf{Trs}}\xspace}

\newcommand{\inccmp}{\ensuremath{\subseteq\!\!\mathsf{Cmp}}\xspace}

\newcommand{\incid}{\ensuremath{\subseteq\!\mathsf{Id}}\xspace}

\newcommand{\incrdt}{\ensuremath{\subseteq\!\mathsf{Rdt}}\xspace}

\newcommand{\inczext}{\ensuremath{\subseteq_0\!\!\mathsf{Ext}}\xspace}

\newcommand{\rori}{\ensuremath{\ror\textsf{I}}\xspace}

\newcommand{\todo}[1]{}

\journal{}

\begin{document}

%

\begin{frontmatter}

\title{Propositional union closed team logics}

 \author{Fan Yang}

  \address{Department of Mathematics and Statistics, PL 68 (Pietari Kalmin katu 5),
00014 University of Helsinki,
Finland\\fan.yang.c@gmail.com}

\begin{abstract}
In this paper, we study several propositional team logics that are closed under unions, including propositional inclusion logic. We show that all these logics are expressively complete, and we introduce sound and complete systems of natural deduction for these logics. We also discuss the locality property and its connection with interpolation in these logics. 
\end{abstract}


\end{frontmatter}

\section{Introduction}

In this paper, we study propositional union closed team logics. These logics are variants of {\em dependence logic}, which was introduced by V\"{a}\"{a}n\"{a}nen \cite{Van07dl} as a non-classical first-order logic for reasoning about dependencies. This framework extends the classical logic by adding new atomic  formulas for charactering dependence and independence between variables. Examples of such atoms are {\em dependence atoms} (giving rise to {\em dependence logic} \cite{Van07dl}),  {\em inclusion atoms} (giving rise to {\em inclusion logic} \cite{Pietro_I/E}) and {\em independence atoms} (giving rise to {\em independence logic} \cite{D_Ind_GV}).   
Dependence logic and its variants adopts the so-called {\em team semantics}, which was introduced by Hodges \cite{Hodges1997a,Hodges1997b}. The basic idea of team semantics is that dependency properties can only manifest themselves in {\em multitudes}.
Thus, formulas of these logics are evaluated under {\em teams}, which, in the propositional context, are {\em sets} of valuations. In particular, a propositional inclusion atom $p\subseteq q$ is said to be true in a team $X$, if every truth value of $p$ in $X$ occurs as a truth value of $q$; in other words, values of $p$ are ``included'' in the values of $q$.

In logics based on team semantics, also called {\em team(-based) logics}, two closure properties are of particular interest: the {\em downwards closure} and the {\em union closure} property.
Dependence logic is {\em closed downwards}, meaning that the truth of a formula on a team  is preserved under taking subteams. In this paper, we focus on propositional team-based logics that are  {\em closed under unions}, meaning that if two teams both satisfy a formula, then their set-theoretic union also satisfies the  formula. Inclusion logic is the first studied union closed team logic \cite{Pietro_I/E}. First-order dependence and inclusion logic can  be translated into existential second-order logic (\textsf{ESO}) \cite{Van07dl,Pietro_I/E}. More precisely, first-order dependence logic characterizes all downwards closed \textsf{ESO}-team properties \cite{KontinenVaa2009}, whereas some but not all union closed \textsf{ESO}-team properties are definable in first-order inclusion logic \cite{inclusion_logic_GH}. While both dependence and inclusion logic are strictly weaker than \textsf{ESO} on their own,  full \textsf{ESO} can be characterized by first-order logic extended with both the downwards closed dependence atoms and the union closed inclusion atoms \cite{Pietro_I/E}. 

Interestingly, the complementary feature of downwards closure and union closure 
properties is also  found on a more basic level in team semantics: Team logics are often defined as conservative extensions of classical logic. The {\em conservativity} is described through the {\em flatness} property of classical formulas, which states that every classical formula is satisfied in a team $X$ if and only if every singleton team of a valuation  in $X$ satisfies the formula, or equivalently, every valuation in $X$ satisfies the formula in the sense of the usual semantics. The property of flatness can actually be decomposed into downwards closure and union closure, in the sense that a formula is flat if and only if it is both downwards closed and union closed (see Fact \ref{flat=0+dc+uc}), assuming that the empty team satisfies the formula (which is often the case for many typical team logics).

%

Understanding the properties of team logics with the downwards closure or union closure property is thus arguably a key theme in the research in team semantics. Compared with the relatively well-understood downwards closed team logics (particularly dependence logic), union closed team logics have received less attention in the literature.  
%
Many properties of inclusion logic and other union closed team logics have not yet been  well-explored, especially on the propositional level. On the first-order level, in 2013, first-order inclusion logic was shown by Galliani and Hella \cite{inclusion_logic_GH}  to be expressively equivalent to positive greatest fixed point logic and thus captures the complexity class \textsf{NP} over finite ordered structures. 
This breakthrough has sparked increasing  interests in  inclusion logic and union closed team logics in general in recent years.
For instance, model-checking games for first-order inclusion logic were developed in \cite{Gradel16,GradelHegselmann16}, first-order consequences of first-order inclusion logic were axiomatized in \cite{YangInc20}, computational complexity and syntactical fragments of first-order inclusion logic were investigated in \cite{Hierarchies_Ind_GHK,Hannula_inc15,HannulaHella19,HannulaKontinen15,Ronnholm_thsis},  a team-based first-order logic characterizing the union closed fragment of existential second-order logic was identified in \cite{HoelzelWilke20}, etc. As for propositional logic, some basic properties of propositional inclusion logic (\PInc) and other union closed team logics were discussed  in \cite{YangVaananen:17PT}. The results in \cite{YangVaananen:17PT} are, however, relatively preliminary, compared with the extensive account of propositional downwards closed team logics in the literature (e.g., \cite{InquiLog,VY_PD}).
There are some recent articles on the expressive power  and computational complexity properties of modal  inclusion logic \cite{HellaKuusistoMeierVirtema17,HellaKuusistoMeierVollmer19,HellaStumpf15} that also cover propositional inclusion logic, but only briefly as a special case. 
The aim of this paper is to provide a  more in-depth account for the logical properties of propositional union closed team logics, including \PInc and another two logics.

One of the two logics is obtained by extending classical propositional logic with a different type of union closed dependency atoms $\mathsf{p}\Upsilon \mathsf{q}$ (with $\mathsf{p},\mathsf{q}$ two sequences of propositional variables), called the {\em anonymity atoms}. These atoms were introduced  by Galliani in \cite{Pietro_thesis} under the name {\em  non-dependence atoms}, as they state an extreme case of the failure of the functional dependence between $\mathsf{p}$ and $\mathsf{q}$. 
Recently V\"{a}\"{a}n\"{a}nen \cite{Vaananen_anonymity19} advocated the anonymity atoms with the motivation of anonymity concerns in data safety (and hence the name): $\mathsf{p}\Upsilon \mathsf{q}$ means ``$\mathsf{p}$ is anonymized with respect to $\mathsf{q}$''. Anonymity atoms also correspond exactly to the {\em afunctional dependencies} studied in database theory (see e.g., \cite{DeBraParedaens82,BraParedaens83}).  A special case $\langle\rangle\Upsilon \mathsf{q}$ of anonymity atoms with the first component being the empty sequence $\langle\rangle$ deserves commenting. Such atoms, denoted also as $\idep(\mathsf{p})$, are also called {\em inconstancy atoms}, as they state  that $\mathsf{p}$ does not have a constant value in the team. First-order logic with inconstancy atoms is known to be equivalent to first-order logic over sentences \cite{Pietro_I/E}, and first-order logic with arbitrary anonymity atoms is equivalent to inclusion logic \cite{Pietro_thesis}. The propositional logic with these atoms (denoted as \PAm and \Picnst) have so far not been studied.

The other union closed team logic we consider is obtained by adding to classical propositional logic a new disjunction $\ror$, called the {\em relevant disjunction}, which  was introduced by R\"{o}nnholm  \cite{Ronnholm_thsis}  as a variant of the standard disjunction $\vee$  in  team semantics. The main difference between the two disjunctions is that a team can satisfy a disjunction $\phi\vee\psi$ when only one disjunct is satisfied, whereas the relevant disjunction $\phi\ror\psi$ requires both disjuncts to be satisfied in a non-void manner (and thus both disjuncts are actually ``relevant"). 
The relevant disjunction is studied in the literature also under the name {\em nonempty disjunction} (e.g., \cite{HellaStumpf15,YangVaananen:17PT}). In \cite{HellaStumpf15}, classical modal logic with $\ror$ was shown to be expressively complete. 

The starting point of this paper is the work in \cite{YangVaananen:17PT}, where classical propositional logic with relevant disjunction (\PU) was shown to be expressively complete. 
Building on the arguments in \cite{YangVaananen:17PT}, we prove in this paper that \PAm as well as \Picnst are expressively complete too. It follows essentially from the argument of \cite{HellaStumpf15} in the context of modal logic that propositional inclusion logic \PInc (with actually a weaker version of inclusion atoms than in \cite{HellaStumpf15}) 
is also expressively complete. 
All of the these union closed logics are thus also expressive equivalent, and they all admit certain disjunctive normal form.

We also provide axiomatizations for \PInc and \PU, which are lacking in the literature. We introduce sound and complete natural deduction systems  for these logics. As with other team logics, these systems do not admit uniform substitution. 
The completeness theorem 
are proved by using the disjunctive normal form of the logics.

In union closed team logics, a metalogical property, the {\em locality} property, deserves particular attention. Locality states that the truth of a formula does not depend on the variables that do not occur in the formula. While this property is often taken for granted in most familiar logics, it is actually a very non-trivial property for team logics, especially for union closed team logics. For example, first-order inclusion logic with the so-called {\em strict semantics} does not satisfy locality \cite{Pietro_I/E}. We give  examples to show that under strict semantics, locality fails for propositional inclusion logic \PInc, as well as \PU and \PAm too. We also discuss a subtle connection between locality and interpolation. It follows from the work of D'Agostino \cite{DAgostino19} in the  modal team logics context that all expressively complete union closed team logics enjoy uniform interpolation. We highlight the subtle and crucial role that locality plays in the poof of \cite{DAgostino19}, and also give an example to illustrate the failure of Craig's interpolation in a fragment of \PU under strict semantics that does not satisfy the locality property.

This paper is organized as follows. In Section 2, we recall the basics of team semantics and define the propositional union closed  team logics we consider in the paper. In Section 3, we show that  these  logics are expressively complete (some of the proofs are already known), and obtain a disjunctive normal form for the logics. Making essential use of this disjunctive normal form, in Section 4, we axiomatize \PU and \PInc as well as a fragment \PIncz of \PInc.  In Section 5, we revisit the property of locality, and highlight its connection with interpolation.
%
%
%
 We conclude and discuss further directions in Section 6.



\section{Preliminaries}\label{sec:pre}

All of the logics we consider in this paper are extensions of classical propositional logic, defined in the team semantics setting. 
Let us now start by recalling the syntax for {\em classical propositional logic} (\CPL). Fix a set $\textsf{Prop}$ of propositional variables. The set of well-formed formulas of \CPL (called {\em classical formulas}) 
are given by the grammar:
\[
    \alpha::= \,p\mid  \bot\mid  \top\mid\neg \alpha\mid (\alpha\wedge\alpha)\mid(\alpha\sor\alpha)
\] 
where $p\in \textsf{Prop}$, and $\bot$ and $\top$ are two constants, called {\em falsum} and  {\em versum}, respectively.
Throughout the paper we reserve the first Greek letters $\alpha,\beta,\gamma,...$ for classical formulas.
As usual, 
 we write $\alpha\to\beta:=\neg\alpha\vee\beta$ and $\alpha\leftrightarrow\beta:=(\alpha\to\beta)\wedge(\beta\to\alpha)$.
We write $ \textsf{Prop}(\alpha)$ for the set of propositional variables occurring in  $\alpha$. We also use the notation $\alpha(\mathsf{N})$ (with $\mathsf{N}\subseteq \textsf{Prop}$ a set of propositional variables) to indicate that the propositional variables occurring in $\alpha$ are among $\mathsf{N}$.

\begin{table}[t]
\begin{center}
\begin{tabular}{|c|c|c|}\hline
&$p$&$q$\\\hline
$v_1$&$0$&$1$\\
$v_2$&$1$&$0$\\
$v_3$&$1$&$1$\\
\hline
\end{tabular}
\end{center}
\caption{A team $X$ with $\textsf{dom}(X)=\{p,q\}$}
\label{tbl:team} 
\end{table}%

Let $\textsf{N}=\{p_1,\dots,p_n\}\subseteq \textsf{Prop}$ be a set of propositional variables. In the standard semantics for \CPL, a classical formula $\alpha(\mathsf{N})$ is evaluated under {\em valuations}, which are functions $v:\mathsf{N}\cup\{\bot,\top\}\to\{0,1\}$ such that $v(\bot)=0$ and $v(\top)=1$. Recall that a valuation $v$ extends naturally to all formulas $\alpha$ of \CPL, and we write $v\models\alpha$ if $v(\alpha)=1$.
In this paper, we adopt {\em team semantics} for classical formulas, in which a classical formula is evaluated under teams. An {\em (\textsf{N}-)team} $X$ is a set of valuations $v:\mathsf{N}\cup\{\bot,\top\}\to\{0,1\}$ with $v(\bot)=0$ and $v(\top)=1$. The set \textsf{N} is called the {\em domain} of the team $X$, denoted as $\textsf{dom}(X)$. 
In particular, the empty set $\emptyset$ is a team (of an arbitrary domain). We often represent a team as a table. For example, \Cref{tbl:team} represents a team $X=\{v_1,v_2\}$ with $\textsf{dom}(X)=\{p,q\}$ consisting of two valuations $v_1$ and $v_2$, defined as
\[v_1(p)=0,~v_1(q)=1,~v_2(p)=1,~v_2(q)=0,\text{ and }v_3(p)=v_3(q)=1.\]

The notion of a classical formula $\alpha$  being \emph{true} on a team $X$ with $\textsf{dom}(X)\supseteq \textsf{Prop}(\alpha)$, denoted by $X\models\alpha$, is defined inductively as follows:
\begin{itemize}
\item $X\models p$ ~~iff~~ for all $v\in X$,  $v(p)=1$.
\item $X\models \bot$ ~~iff~~ $X=\emptyset$.
\item $X\models \top$ always holds.
\item $X\models \neg \alpha$ ~~iff~~ for all $v\in X$,  $\{v\}\not\models\alpha$.
\item $X\models\alpha\wedge\beta$ ~~iff~~ $X\models\alpha$ and $X\models\beta$.
\item $X\models\alpha\sor\beta$ ~~iff~~  there exist subteams $Y,Z\subseteq X$ such that $X=Y\cup Z$, 
\(Y\models\alpha\text{ and }Z\models\beta.\)
\end{itemize}
For any set $\Gamma\cup\{\alpha\}$ of formulas, we write $\Gamma\models\alpha$ if for all teams $X$ with $\textsf{dom}(X)\supseteq \bigcup_{\gamma\in\Gamma}\textsf{Prop}(\gamma)\cup\textsf{Prop}(\alpha)$, $X\models\gamma$ for all $\gamma\in\Gamma$ implies $X\models\alpha$. We write simply $\alpha\models\beta$ for $\{\alpha\}\models\beta$. If both $\alpha\models\beta$ and $\beta\models\alpha$, we write $\alpha\equiv\beta$ and say that $\alpha$ and $\beta$ are {\em semantically equivalent}.

It is easy to verify (by a straightforward induction) that \CPL-formulas have the  {\em locality property}, {\em empty team property}, {\em union closure property} and {\em downwards closure property}, that is, for any \CPL-formula $\alpha$, the following holds:
%
%
\begin{description}
\item[Empty Team Property:] $\emptyset\models\alpha$ holds; 
\item[Union Closure:] $X\models\alpha$ and $Y\models\alpha$ imply $X\cup Y\models\alpha$;
\item[Downwards Closure:] $X\models\alpha$ and $Y\subseteq X$ imply $Y\models\alpha$.
\end{description}
These three properties together are (easily shown to be) equivalent to the {\em flatness property}:
\begin{description}
\item[Flatness:] $X\models\alpha$ if and only if $\{v\}\models\alpha$ for all $v\in X$.
\end{description}

\begin{fact}\label{flat=0+dc+uc}
A formula is flat if and only it satisfies the empty team property and is both union closed and downwards closed.
\end{fact}

Moreover, an easy inductive proof shows that the truth of a classical formula $\alpha$ on singleton teams $\{v\}$ coincides with its truth  on the single valuations $v$ in the usual sense, namely,
\begin{equation}
\{v\}\models\alpha\text{ if and only if }v\models\alpha.
\end{equation}
Putting these observations together, we obtain the following fact, which will also serve as our working team semantics for classical formulas:
\begin{fact}\label{cpl_singleton_single}
For any classical formula $\alpha$, any team $X$ with $\textsf{dom}(X)\supseteq \textsf{Prop}(\alpha)$,
\[X\models\alpha~~\text{ iff }~~v\models\alpha\text{ for all }v\in X.\]
\end{fact}

Intuitively, the above fact (or essentially the flatness of classical formulas) means that the team semantics for classical formulas collapses to the usual (single-valuation) semantics. In this sense, team semantics is conservative over classical formulas. Such a conservativity as reflected in \Cref{cpl_singleton_single} also justifies our definition of the team semantics for classical formulas. For instance, our semantic clause ``$X\models\alpha\vee\beta$'' for a disjunction being true {\em on the team level} states that the team $X$ can be split into two subteams $Y$ and $Z$ such that each disjunct is true in one of the two subteams. This (by \Cref{cpl_singleton_single}) is the same as stating that the disjunction $\alpha\vee\beta$ is true {\em locally} under every valuation $v$ in $X$, that is, every valuation $v$ in $X$ satisfies at least one of the disjuncts, resulting in a natural split of the team $X$ into two subteams $Y$ and $Z$ with every element in $Y$ making $\alpha$ true and every element in $Z$ making $\beta$ true. By \Cref{cpl_singleton_single} again, this means that $\alpha$ is true at $Y$ on the team level, and $\beta$ is true at $Z$ on the team level.



Recall that for classical formulas under the usual (single-valuation) semantics, when a formula $\alpha$ is evaluated on a valuation $v$, the truth of $\alpha$ depends only on how the valuation $v$ evaluates the propositional variables that actually occurs in $\alpha$, and it is independent of how the other propositional variables are evaluated. That is, for any formula $\alpha(\mathsf{N})$, if $v$ and $u$ are two valuations with $v\upharpoonright \mathsf{N}=u\upharpoonright \mathsf{N}$, then $v\models\alpha$ iff $u\models\alpha$. 
This is called the {\em locality property}. While this property is often taken for granted for most logics with the usual semantics, as we will discuss in this paper, in the team semantics setting, the locality property is a nontrivial property. We now give the definition of the property in the team semantics setting.


\begin{definition}
Let $\mathsf{N}\subseteq\mathsf{Prop}$ be a set of propositional variables. A formula $\phi(\textsf{N})$ is said to satisfy the {\em locality property} if for any  teams $X$ and $Y$ such that $\textsf{dom}(X),\textsf{dom}(Y)\supseteq\mathsf{N}$ and $X\upharpoonright \mathsf{N}=Y\upharpoonright \mathsf{N}$, it holds that
 \[X\models\phi(\textsf{N})\iff Y\models\phi(\textsf{N}),\]
 where for any team $Z$ with $\mathsf{dom}(Z)\supseteq\mathsf{N}$, we write
\(Z\upharpoonright\mathsf{N}=\{v\upharpoonright \mathsf{N}\mid v\in Z\}.\)
\end{definition}

Thanks to the flatness of classical formulas, we can, nevertheless, obtain the locality property for classical formulas as an immediate corollary of \Cref{cpl_singleton_single}.
\begin{proposition}
Classical formulas satisfy the locality property.
\end{proposition}
\begin{proof}
Let $\alpha(\textsf{N})$ be a classical formula. For any teams $X$ and $Y$ with $\textsf{dom}(X),\textsf{dom}(Y)\supseteq\mathsf{N}$ and $X\upharpoonright \mathsf{N}=Y\upharpoonright \mathsf{N}$, by \Cref{cpl_singleton_single} and the locality property of classical formulas under the usual (single-valuation) semantics, we have that
\[X\models\alpha\iff \forall v\in X: v\models\alpha(\mathsf{N})\iff  \forall u\in Y: u\models\alpha(\mathsf{N})\iff Y\models\alpha.\]
\end{proof}


We now extend \CPL to three non-flat but union closed team-based logics. Consider a new disjunction $\ror$, called {\em relevant disjunction}, 
and  atomic formulas of the form $a_1\dots a_k\subseteq b_1\dots b_k$ with each $a_i,b_i\in \textsf{Prop}\cup\{\bot,\top\}$, called {\em inclusion atoms}, and of the form $p_{1}\dots p_{k}\Upsilon q_1\dots q_m$ with each $p_i,q_j\in\textsf{Prop}$, called {\em anonymity atoms}.  Inclusion and anonymity atoms are often represented as $\mathsf{a}\subseteq\mathsf{b}$ and $\mathsf{p}\Upsilon\mathsf{q}$ with letters $\mathsf{a},\mathsf{b},\mathsf{p},\mathsf{q}$ in serif font standing for sequences of propositional variables or constants of certain lengths. 
Define the  logic \PUw as the extension of \CPL  by adding relevant disjunction $\ror$, and negation $\neg$ is allowed to occur only in front of classical formulas, that is, formulas of \PUw are formed by the grammar:
\[
    \phi::= \,p\mid \bot\mid\top\mid\neg \alpha\mid (\phi\wedge\phi)\mid(\phi\sor\phi)\mid(\phi\ror\phi)   
\] 
where $\alpha$ stands for an arbitrary classical formula. Similarly for the extensions 
\PInc and \PAm of \CPL obtained by adding the inclusion atoms  $\mathsf{a}\subseteq\mathsf{b}$ and anonymity atoms $\mathsf{p}\Upsilon\mathsf{q}$, respectively, where, again, negation $\neg$ is allowed to occur only in front of classical formulas.

Define the team semantics of the new connective and atoms as follows:
\begin{itemize}
\item $X\models\phi\ror\psi$ ~~iff~~ $X=\emptyset$ or there exist nonempty subteams $Y,Z\subseteq X$ such that $X=Y\cup Z$, $Y\models\phi$ and $Z\models\psi$.
\item $X\models \mathsf{a}\subseteq \mathsf{b}$ ~~iff~~  for all $v\in X$,  there exists $u\in X$ such that $v(\mathsf{a})=u(\mathsf{b})$.
\item $X\models \mathsf{p}\Upsilon \mathsf{q}$ ~~iff~~  for all $v\in X$, there exists $u\in X$ such that $v(\mathsf{p})=u(\mathsf{p})$ and $v(\mathsf{q})\neq u(\mathsf{q})$.
\end{itemize}

It is easy to see that formulas in \PUw, \PInc and \PAm  are not necessarily downwards closed, and thus  not necessarily flat. For instance, for the team $X$ from \Cref{tbl:team}, we have that $X\models p\ror q$. However, for the subteam $Y=\{v_2\}$ of $X$, we have that $Y\not\models p\ror q$, since no nonempty subteam of $Y$ makes $q$ true. The reader can also easily verify that for the team $Z=\{v_2,v_3\}$, we have $Z\models p\subseteq q$ and $Z\models p\Upsilon q$, whereas for the subteam $Y=\{v_2\}$ of $Z$, we have $Y\not\models p\subseteq q$ and $Y\not\models p\Upsilon q$. On the other hand, formulas in all these three logics satisfy the union closure property, as well as the empty team property and the locality property.

\begin{lemma}\label{empty_prop_locality_proof}
Formulas in the logics \PUw, \PInc and \PAm satisfy the empty team property, the union closure property and the locality property. 
\end{lemma}
\begin{proof}
The lemma is proved by a straightforward induction on the complexity of formulas $\phi$ in the logics. We only give the proof details for the logic \PUw and for the case $\phi=\psi\ror\chi$. By definition, $\emptyset\models\psi\ror\chi$ trivially holds. For union closure, suppose $X\models\psi\ror\chi$ and $Y\models\psi\ror\chi$. If $X=\emptyset$ or $Y=\emptyset$, then $X\cup Y\models\psi\ror\chi$ trivially holds. Assume now $X,Y\neq\emptyset$. Then there are nonempty teams $X_0,X_1\subseteq X$ and $Y_0,Y_1\subseteq Y$ such that $X=X_0\cup X_1$, $Y=Y_0\cup Y_1$, $X_0\models\psi$, $Y_0\models\psi$, $X_1\models\chi$ and $Y_1\models\chi$. By induction hypothesis, we obtain $X_0\cup Y_0\models\psi$ and $X_1\cup Y_1\models\chi$. Clearly, $X_0\cup Y_0,X_1\cup Y_1\neq\emptyset$ and $X\cup Y=(X_0\cup X_1)\cup(Y_0\cup Y_1)=(X_0\cup Y_0)\cup (X_1\cup Y_1)$. Hence, we conclude that $X\cup Y\models\psi\ror\chi$.

Lastly, we verify the locality property for $\phi(\mathsf{N})=\psi\ror\chi$. Suppose $X,Y$ are teams with $\mathsf{dom}(X),\mathsf{dom}(Y)\supseteq \mathsf{N}$ and $X\upharpoonright \mathsf{N}=Y\upharpoonright \mathsf{N}$. If $X=\emptyset$, then $Y=\emptyset$, and we have $X\models\psi\ror\chi$ iff $Y\models\psi\ror\chi$ by the empty team property. Now assume that $X\neq\emptyset$, and thus $Y\neq\emptyset$. If $X\models\psi\ror\chi$, then there exist nonempty teams $X_0,X_1\subseteq X$ such that $X=X_0\cup X_1$, $X_0\models\psi$ and $X_1\models\chi$. Consider 
\[Y_0=\{v\in Y\mid v\upharpoonright \mathsf{N}\in X_0\upharpoonright \mathsf{N}\}\text{ and }Y_1=\{v\in Y\mid v\upharpoonright \mathsf{N}\in X_1\upharpoonright \mathsf{N}\}.\]
Since $X_0,X_1\neq \emptyset$ and $X\upharpoonright \mathsf{N}=Y\upharpoonright \mathsf{N}$, we must have that $Y_0,Y_1\neq\emptyset$. It is also not hard to verify that 
\begin{equation}\label{locality_thm_eq1}
Y_0\upharpoonright \mathsf{N}=X_0\upharpoonright \mathsf{N}, ~~Y_1\upharpoonright \mathsf{N}=X_1\upharpoonright \mathsf{N},~\text{ and }Y_0\cup Y_1=Y.
\end{equation}
Now, by induction hypothesis and the first two equations of (\ref{locality_thm_eq1}), we obtain $Y_0\models \psi(\mathsf{N})$ and $Y_1\models\chi(\mathsf{N})$. Finally,  the last equation of (\ref{locality_thm_eq1}) then gives $Y\models\psi\ror\chi$.
\end{proof}

Note the similarity and difference between the semantic clauses of $\sor$ and $\ror$: They both state that the team in question can be split into two subteams, each satisfying one of the disjuncts. The relevant disjunction $\ror$ requires in addition that the two subteams must be nonempty, as long as the starting team is nonempty. In particular, when applied to classical formulas, a (flat) disjunction $\alpha\vee\beta$ being true on a team $X$ means that either disjunct is true on each valuation $v$ in the team $X$ {\em locally}, while a relevant disjunction $\alpha\ror\beta$ being true on the same team $X$ requires, in addition to the local truth of the disjuncts, also that each disjunct is actually true on some valuations (and thus both disjuncts are considered ``relevant" for the truth of the disjunction).
For an illustration of the two different disjunctions, consider the following two sentences in natural language:
\begin{quotation}
{\em (a) The coin lands heads or tails. }

{\em (b) Either the moon is made of green cheese or it rained today.}
\end{quotation}
Consider also the team $X_a$ of the records of certain times of coin tossing, and the team $X_b$ of the meteorological and astronomical reports of a certain year. If we interpret the two disjunctions in (a) and (b) using the disjunction $\vee$ from classical logic (i.e., (a) is understood as $h\vee t$ and (b) as $m\vee r$), the two sentences are then both true in the relevant teams $X_a$ and $X_b$. If we interpret (a) and (b) using the relevant disjunction, as $h\ror t$ and $m\ror r$ instead, then $m\ror r$ will fail in the team $X_b$, and $h\ror t$ is possible to fail in the team $X_a$ if the coin is a bias one.  

Closely related is another disjunction that we shall call the {\em global disjunction} (also known in the literature by the name  {\em intuitionistic disjunction} or {\em Boolean disjunction} or {\em classical disjunction}), defined as
\begin{itemize}
\item $X\models\phi\vvee\psi$ ~~iff~~ $X\models\phi$ or $X\models\psi$.
\end{itemize}
The global disjunction states that either disjunct is true {\em globally} in the team in question.
It is easy to verify that $\phi\vee\psi\equiv\phi\vvee\psi\vvee(\phi\ror\psi)$. The global disjunction does not, however, preserves union closure, as e.g., $p\vvee q$ is clearly not closed under unions. We thus do not consider the global disjunction in this paper. 

Another related logical constant is the unary operator $\might$, called the {\em might modality}, whose team semantics is defined as
\begin{itemize}
\item $X\models\might\phi$ ~~iff~~  $X=\emptyset$ or there exists a nonempty subteam $Y\subseteq X$ such that $Y\models\phi$.
\end{itemize}
It is easy to verify that the might modality preserves union closure.
It was also observed in  \cite{HellaStumpf15} that the relevant disjunction $\ror$ and the might operator $\might$ are inter-definable, as
\(\might\phi\equiv \phi\ror\top\text{ and }\phi\ror\psi\equiv (\phi\sor\psi)\wedge \might\phi\wedge\might\psi. \)
We say that a team-based logic $\LL_1$ is expressively weaker than another team-based logic $\LL_2$, denoted as $\LL_1\leq \LL_2$, if  for every $\LL_1$-formula $\phi$, there exists an $\LL_2$-formula $\psi$ such that $\phi\equiv\psi$. If both $\LL_1\leq \LL_2$ and $\LL_2\leq \LL_1$, then we write $\LL_1\equiv \LL_2$ and say that $\LL_1$ and $\LL_2$ are  {\em expressively equivalent}. Clearly,  $\PU\equiv\PML$ for  the extension \PML of \CPL with the unary might modality $\might$. In this paper we are more interested in binary connectives and atoms;  the might modality is left for future research.

An inclusion atom $\mathsf{a}\subseteq\mathsf{b}$ as we defined can take the two atoms $\bot$ and $\top$ as arguments. It thus has a more relaxed syntax than the standard one in the literature where the arguments $a_i,b_i$ can only be propositional variables. Let us  point out that the standard version of inclusion logic \PInc with this standard syntax of inclusion atoms is strictly weaker in expressive power than our version of \PInc. To see why, consider our inclusion atom $\top\subseteq p$ in one variable. 
To express this inclusion atom in the standard version of \PInc, by the locality property, it is sufficient to consider formulas in the only variable $p$. Modulo equivalence, the only such classical formulas are $\top,\bot,p,\neg p$, and the only inclusion atom with merely the propositional variable $p$ is $p\subseteq p$, which is equivalent to $\top$. All these formulas are flat, and thus are not equivalent to the non-flat inclusion atom $\top\subseteq p$. We will show in this paper that our version of \PInc is actually expressively complete. Our proof applies essentially the same argument as in \cite{HellaStumpf15} in the context of modal inclusion logic, which has an even more relaxed syntax for inclusion atoms $\mathsf{a}\subseteq\mathsf{b}$, for which the arguments $a_i,b_i$ are allowed to be arbitrary classical formulas. Such inclusion atoms are known in the literature as the {\em extended inclusion atom}. It follows from  \cite{HellaStumpf15} that propositional inclusion logic \PInc with extended inclusion atoms is expressively complete, and is thus expressively equivalent to our version of \PInc with the relatively less general inclusion atoms as we defined. Our choice of the syntax of \PInc, or of inclusion atoms, thus enables us to obtain an expressively complete logic with minimal modification to the standard syntax of \PInc.

The anonymity atom $\mathsf{p}\Upsilon\mathsf{q}$, also known as {\em non-dependence atom}, states an extreme case of the negation of the functional dependence between $\mathsf{p}$ and $\mathsf{q}$: For {\em every} valuation $v$ in the team $X$ in question, there is a witness $u$ in $X$ with respect to $v$ that witnesses the failure of the said functional dependence. Recently V\"{a}\"{a}n\"{a}nen \cite{Vaananen_anonymity19} also used this atom to express the property ``$\mathsf{p}$ is anonymized with respect to $\mathsf{q}$'', and hence the name.
Such defined anonymity atoms $\mathsf{p}\Upsilon\mathsf{q}$ also corresponds exactly to the {\em afunctional dependencies} studied in database theory (see e.g., \cite{DeBraParedaens82,BraParedaens83}). 
We write $\idep(\mathsf{p})$ for the anonymity atom $\langle\rangle\Upsilon \mathsf{p}$ whose left component is the empty sequence $\langle\rangle$,  and call such an atom {\em inconstancy atom}. Clearly, the semantics clause of the inconstancy atom $\idep(\mathsf{p})$ 
reduces to
\begin{itemize}
\item $X\models\idep(\mathsf{p})$ ~~iff~~ either $X=\emptyset$ or there exist $v,u\in X$ such that $v(\mathsf{p})\neq u(\mathsf{p})$.
\end{itemize}
Intuitively, $\idep(\mathsf{p})$ states that the sequence $\mathsf{p}$ of propositional variables does not have a constant value in the team in question. It is easy to verify that inconstancy atoms with multiple arguments are definable in terms of those with single arguments:
\[\idep(p_1\dots p_n)\equiv\idep(p_1)\vee\dots\vee\idep(p_n).\]
In addition, inconstancy atoms with single arguments are  easily definable in terms of relevant disjunction: $\idep(p)\equiv p\ror\neg p$.
 
 Recall that an atom of a dual flavor is the {\em constancy atom} $\dep(p)$ which states that $p$ has a constant value in the team:
\begin{itemize}
\item $X\models\dep(p)$ ~~iff~~ for all $v,u\in X$, $v(p)=u(p)$.
\end{itemize}
Constancy atoms are clearly downwards closed. Dually, inconstancy atoms $\lambda$ are clearly {\em upwards closed}, meaning that 
$X\models\lambda$ and $Y\supseteq X$ imply $Y\models\lambda$.
Upwards closure clearly implies union closure. 

We call the inclusion atoms $\mathsf{x}\subseteq\mathsf{a}$ with  $x_i\in\{\bot,\top\}$  for each $i$ {\em primitive inclusion atoms}. For instance, $\top\bot\subseteq pq$ and $\bot \top\subseteq \top p$ are primitive inclusion atoms, whereas $p\subseteq q$, $q\subseteq\top$ are not.
Interestingly, primitive inclusion atoms are also upwards closed. Denote by \Picnst and \PIncz the logics extended from \CPL by adding, respectively, inconstancy atoms with single arguments and primitive inclusion atoms. Arbitrary formulas in these sublogics of \PInc and \PAm are, however, not in general upwards closed, as, e.g.,  already the propositional variable $p$ is not upward closed.



\section{Expressive Completeness and normal form}\label{sec:expr_comp}

In this section, we study the expressive power of the logics \PUw, \PInc, \PIncz, \PAm and \Picnst we introduced. The logic $\PUw$ was proved in \cite{YangVaananen:17PT} to be expressively complete. Building on this result, we show that the logics \Picnst and \PAm are also expressively complete. It was proved in \cite{HellaStumpf15} that modal inclusion logic with extended inclusion atoms (which allow arbitrary classical formulas to occur as arguments) is expressively complete, from which it follows essentially that propositional inclusion logic  \PInc and its variant \PIncz are also expressively complete. We recast here the argument of \cite{HellaStumpf15} in our propositional setting to give a detailed proof of the expressive completeness of \PIncz and \PInc. As an immediate corollary, all these five logics are thus expressively equivalent. From the proof of these expressive completeness results, we will also obtain normal forms for formulas in these logics. The (disjunctive) normal form of \PUw has already been introduced in \cite{YangVaananen:17PT}. We will discuss the normal forms for the other logics, as well as normal forms for inclusion atoms and anonymity atoms. The normal forms of these logics will play a crucial role in the axiomatization of the logics in Section 4. The expressive completeness is also required for obtaining the interpolation theorem to be discussed in Section 5.



Let us start by giving formal definitions of the relevant terminologies. A {\em team property} $\mathsf{P}$ is a set of teams over certain domain $\mathsf{N}\subseteq \mathsf{Prop}$. For any formula $\phi(\mathsf{N})$ in the language of any of the above logics, the set 
\[\llbracket\phi\rrbracket_{\mathsf{N}}=\{X\subseteq2^\mathsf{N}: X\models\phi\}\]
of $\mathsf{N}$-teams that satisfy $\phi(\mathsf{N})$ is a team property (over $\mathsf{N}$), where $2^\mathsf{N}$ stands for the set of all valuations $v:\mathsf{N}\cup\{\bot,\top\}\to 2=\{0,1\}$. Clearly, by locality, for any two formulas $\phi(\mathsf{N})$ and $\psi(\mathsf{N})$, $\llbracket\phi\rrbracket_{\mathsf{N}}=\llbracket\psi\rrbracket_{\mathsf{N}}$ implies $\phi\equiv\psi$. For any formula $\phi$ in any of the five union closed logics we consider (i.e., \PUw, \PInc, \PIncz, \PAm and \Picnst), the set $\llbracket\phi\rrbracket_{\mathsf{N}}$ clearly contains the empty team $\emptyset$, and is closed under unions, i.e., $X,Y\in \llbracket\phi\rrbracket_{\mathsf{N}}$ implies $X\cup Y\in \llbracket\phi\rrbracket_{\mathsf{N}}$. 
Let $\mathbb{P}$ be a collection of team properties over some domains. 
For $\mathsf{N}\subseteq \textsf{Prop}$, we write
\[\mathbb{P}_{\mathsf{N}}=\{\mathsf{P}\in \mathbb{P}: \mathsf{P}\text{ is a set of $\mathsf{N}$-teams}\}\]
for the class of team properties over $\mathsf{N}$ that are in $\mathbb{P}$. We are now ready to give the definition for the notion of expressive completeness.
\begin{definition}[expressive completeness]\label{expressive_comp_df}
  We say that a team-based logic \LL \emph{characterizes}  $\mathbb{P}$, or \LL is \emph{expressively complete} in $\mathbb{P}$, if for every set $\mathsf{N}\subseteq \textsf{Prop}$ of propositional variables, 
\[\mathbb{P}_{\mathsf{N}}=\{\llbracket \phi\rrbracket_{\mathsf{N}}:~\phi\text{ is an \LL-formula with }\textsf{Prop}(\phi)=\mathsf{N}\}.\] 
\end{definition}

That is, if \LL is expressively complete in $\mathbb{P}$, then  for every $\mathsf{N}$, every \LL-formula $\phi(\mathsf{N})$ defines a team property $\llbracket \phi\rrbracket_{\mathsf{N}}$ that belongs to $\mathbb{P}_{\mathsf{N}}\subseteq\mathbb{P}$, and conversely, every team property $\mathsf{P}\in \mathbb{P}_{\mathsf{N}}$ is definable by some formula $\phi(\mathsf{N})$  in \LL. Clearly, if both $\LL_1$ and $\LL_2$ are expressively complete in some class $\mathbb{P}$, then $\LL_1$ and $\LL_2$ are expressively equivalent, i.e., $\LL_1\equiv\LL_2$, since for every $\LL_1$-formula $\phi(\mathsf{N})$, the team property  $\llbracket \phi\rrbracket_{\mathsf{N}}\in \mathbb{P}_{\mathsf{N}}$ is definable by some $\LL_2$-formula $\psi(\mathsf{N})$, namely $\llbracket \phi\rrbracket_{\mathsf{N}}=\llbracket \psi\rrbracket_{\mathsf{N}}$ or $\phi\equiv\psi$; and vice versa.



Let $\mathbb{F}$ denote the collection of all flat team properties (i.e., properties $\mathsf{P}$ satisfying $X\in \mathsf{P}$ iff $\{v\}\in \mathsf{P}$ for all $v\in X$). It was proved in \cite{YangVaananen:17PT} that classical propositional logic \CPL is expressively complete in $\mathbb{F}$. We now recall briefly also the proof of this fact from \cite{YangVaananen:17PT}, as a crucial formula $\Theta_X$ used in this proof will play an important role in the main theorem (\Cref{PT_exp_pw}) in this section.



\begin{theorem}[\cite{YangVaananen:17PT}]\label{cpl_excmp}
\CPL is expressively complete in $\mathbb{F}$.
\end{theorem}
\begin{proof}
We only give a sketch of the proof. For any \CPL-formula $\alpha(\mathsf{N})$, since $\alpha$ is flat, we have $\llbracket \alpha\rrbracket_{\mathsf{N}}\in\mathbb{F}$. Conversely, for any $\mathsf{N}$-team property $\mathsf{P}\in \mathbb{F}_{\mathsf{N}}$, putting $X=\bigcup \mathsf{P}$, we show that $\llbracket \Theta_X\rrbracket_{\mathsf{N}}=\mathsf{P}$, where 
\[\Theta_X:=\bigsor_{v\in X}(p_{1}^{v(1)}\wedge\dots\wedge p_{n}^{v(n)})\] 
with $v(i)$  short for $v(p_i)$, $p_i^1:=p_i$, $p_i^0=\neg p_i$ and $\bigvee\emptyset:=\bot$.  Indeed, each disjunct in the formula $\Theta_X$ clearly defines a valuation $v$ in the team $X$, in the sense that for any $\mathsf{N}$-team $Y$, 
\begin{equation}\label{cpl_excmp_eq1}
Y\models p_{1}^{v(1)}\wedge\dots\wedge p_{n}^{v(n)} \iff Y\subseteq\{v\}.
\end{equation}
It then follows that for any $\mathsf{N}$-team $Y$, 
\begin{equation}\label{cpl_excmp_eq2}
Y\models \Theta_X\iff Y=\bigcup_{v\in X}Y_v\text{ for some }Y_v\subseteq\{v\} \iff Y\subseteq X.
\end{equation}
Thus, since $\mathsf{P}$ is flat, we have  $Y\models\Theta_X$ iff $Y\subseteq X= \bigcup \mathsf{P}\in \mathsf{P}$ iff $Y\in\mathsf{P}$.
\end{proof}

The expressive completeness of \CPL immediately implies the following characterization of classical formulas in the logics we consider in this paper.

\begin{corollary}\label{flat_char}
A  formula of any of the  logics \PUw, \PInc, \PIncz, \PAm and \Picnst is flat iff it is equivalent to a classical formula.
\end{corollary}
\begin{proof}
The right to left direction is obvious. For the other direction, let $\phi(\mathsf{N})$ be a flat formula  in any of the five logics. The team property $\llbracket \phi\rrbracket_{\mathsf{N}}$  is clearly also flat. Then, by the expressive completeness of \CPL, there is a classical formula $\alpha(\mathsf{N})$ such that $\llbracket \phi\rrbracket_{\mathsf{N}}=\llbracket \alpha\rrbracket_{\mathsf{N}}$, which implies $\phi\equiv\alpha$ by locality.
\end{proof}


Denote by $\mathbb{P}^{\dot{\overline{\cup}}}$ the collection of all union closed team properties which contain the empty team. Clearly, for any formula $\phi(\mathsf{N})$ in the union closed team logics we introduced, $\llbracket \phi\rrbracket_{\mathsf{N}}\in \mathbb{P}^{\dot{\overline{\cup}}}$. We will see in this section that the other direction holds for all five union closed team logics we consider in the paper. First, let us recall that this result was proved for \PU already in \cite{YangVaananen:17PT}.

\begin{theorem}[\cite{YangVaananen:17PT}]\label{PU_excmp}
\PU is expressively complete in $\mathbb{P}^{\dot{\overline{\cup}}}$
\end{theorem}
\begin{proof}
See \cite{YangVaananen:17PT} for the detailed proof. We now give a sketch of the proof. 
It suffices to show that for every team property $\mathsf{P}\in \mathbb{P}^{\dot{\overline{\cup}}}_{\mathsf{N}}$, we can find a \PU formula $\phi(\mathsf{N})$ such that $\mathsf{P}=\llbracket\phi\rrbracket_{\mathsf{N}}$. 
For any team $X$ (i.e., a set of valuations $v$) with $\mathsf{dom}(X)=\mathsf{N}=\{p_1,\dots,p_n\}$, define  a \PU-formula 
\begin{equation}\label{eq_psi_x}
\Psi_X:=\bigror_{v\in X}(p_{1}^{v(1)}\wedge\dots\wedge p_{n}^{v(n)}),
\end{equation}
where, again, $v(i)$ is short for $v(p_i)$, $p_i^1:=p_i$, $p_i^0=\neg p_i$ and \,$\bigroril\emptyset:=\bot$. 
Since (\ref{cpl_excmp_eq1}) holds for each $\ror$-disjunct in the formula $\Psi_X$, the formula $\Psi_X$  characterizes the team $X$ modulo the empty team, in the sense that  for any $\mathsf{N}$-team $Y$
\begin{equation}\label{eq_psi_x_prop}
Y\models\Psi_X\iff Y=\emptyset\text{, or }Y=\bigcup_{v\in X}Y_v\text{ and }Y_v=\{v\}\iff Y=\emptyset\text{ or }Y= X.
\end{equation}
Finally, we show $\mathsf{P}=\llbracket\bigsor_{X\in \mathsf{P}}\Psi_X\rrbracket_{\mathsf{N}}$. The inclusion ``$\subseteq$'' is immediate. For the other inclusion ``$\supseteq$", for any $\mathsf{N}$-team $Y$, if $Y\models\bigsor_{X\in \mathsf{P}}\Psi_X$, then $Y=\bigcup_{X\in \mathsf{P}}Y_X$ for some $Y_X$ satisfying $Y_X\models \Psi_X$. It follows from (\ref{eq_psi_x_prop}) that each $Y_X=X$ or $Y_X=\emptyset$. Thus, $Y=\bigcup\mathsf{P}'$ for some (possibly empty) subclass $\mathsf{P}'\subseteq\mathsf{P}$. Now, $Y\in \mathsf{P}$ follows from the fact that $\mathsf{P}$ contains the empty team $\emptyset$ and is closed under unions.
\end{proof}

The proof of the above theorem gives rise to a disjunctive normal form for formulas in \PU.

\begin{corollary}[\cite{YangVaananen:17PT}, Normal Form]\label{nf-pu}
Every \PU formula $\phi(\mathsf{N})$ is equivalent to a formula for the form $\bigvee_{X\in \mathcal{X}}\Psi_X$ for some collection $\mathcal{X}$ of $\mathsf{N}$-teams.
\end{corollary}
\begin{proof}
Put $\mathsf{P}=\llbracket\phi\rrbracket_{\mathsf{N}}$. Since $\mathsf{P}\in\mathbb{P}^{\dot{\overline{\cup}}}_{\mathsf{N}}$, by the proof of Theorem \ref{PU_excmp}, we have $\llbracket\phi\rrbracket_{\mathsf{N}}=\mathsf{P}=\llbracket\bigsor_{X\in \mathsf{P}}\Psi_X\rrbracket_{\mathsf{N}}$, and thus $\phi\equiv\bigsor_{X\in \mathsf{P}}\Psi_X$ follows from locality.
\end{proof}

Another immediate corollary of the expressive completeness of \PU (in $\mathbb{P}^{\dot{\overline{\cup}}}$) is that all union closed team logics with the empty team property (including \PUw, \PInc, \PIncz, \PAm and \Picnst) are compact, as  \PUw was shown in \cite{YangVaananen:17PT} to be compact. 

\begin{corollary}[Compactness]\label{compactness_thm}
Let \LL be a team-based logic that is closed under unions and has the empty team property. Then \LL is compact, that is, for any set  $\Gamma\cup\{\phi\}$ of \LL-formulas, if $\Gamma\models\phi$, then there exists a finite set $\Gamma_0\subseteq\Gamma$ such that $\Gamma_0\models\phi$.
In particular, \PUw, \PInc, \PIncz, \PAm and \Picnst are compact.
\end{corollary}
\begin{proof}
 Since  \LL is closed under unions and has the empty team property, for every \LL-formula $\psi(\mathsf{N})\in \Gamma\cup\{\phi\}$, the $\mathsf{N}$-team property $\llbracket\psi\rrbracket_{\mathsf{N}}$ that $\psi$ defines belongs to $\mathbb{P}^{\dot{\overline{\cup}}}_{\mathsf{N}}$. Since \PU is expressively complete in $\mathbb{P}^{\dot{\overline{\cup}}}$, there exists a \PU-formula $\psi^\ast(\mathsf{N})$ such that $\llbracket\psi^\ast\rrbracket_{\mathsf{N}}=\llbracket\psi\rrbracket_{\mathsf{N}}$ and thus $\psi\equiv\psi^\ast$. Since $\Gamma\models\phi$, we have $\Gamma^\ast\models\phi^\ast$, where $\Gamma^\ast=\{\gamma^\ast\mid \gamma\in \Gamma\}$. Now, since \PUw is compact (by \cite{YangVaananen:17PT}), there exists a finite set $\Gamma_0\subseteq\Gamma$ such that $\Gamma_0^\ast\models\phi^\ast$, which then gives $\Gamma_0\models\phi$.
\end{proof}


Building on \Cref{PU_excmp} and its proof, we now show that all the other union closed team logics \PIncz, \PInc,  \Picnst and \PAm  are also expressively complete in $\mathbb{P}^{\dot{\overline{\cup}}}$. The proofs below for the expressive completeness of \PIncz and \PInc are essentially an adaptation of a similar one for the expressive completeness of modal inclusion logic given in \cite{HellaStumpf15}.



\begin{theorem}\label{PT_exp_pw} 
The logics \PIncz, \PInc,  \Picnst and \PAm are all expressively complete in $\mathbb{P}^{\dot{\overline{\cup}}}$. In particular, 
$\PUw\equiv\PIncz\equiv\PInc\equiv\Picnst\equiv\PAm$. 
\end{theorem}
\begin{proof}
Since \PIncz and \Picnst are sublogics of \PInc and \PAm, respectively, we only need to show that \PIncz and \Picnst are expressively complete in $\mathbb{P}^{\dot{\overline{\cup}}}$. That is to show that for every set $\mathsf{N}$ of propositional variables, for every $\mathsf{N}$-team property $\mathsf{P}\in \mathbb{P}^{\dot{\overline{\cup}}}$, there is a formula $\phi(\mathsf{N})$ in \PIncz and in \Picnst such that $\mathsf{P}=\llbracket\phi\rrbracket_{\mathsf{N}}$. Now, by the proof of \Cref{PU_excmp}, we know that $\mathsf{P}=\llbracket\bigsor_{X\in \mathsf{P}}\Psi_X\rrbracket_{\mathsf{N}}$, where each $\Psi_X$ is a \PU-formula. We would thus be done if for every $\mathsf{N}$-team $X\in\mathsf{P}$, we can find a \PIncz formula $\eta(\mathsf{N})$  and a \Picnst-formula $\psi_X(\mathsf{N})$ such that $\eta_X\equiv\Psi_X\equiv\psi_X$. 

\vspace{0.5\baselineskip}

We first construct for every $\mathsf{N}$-team $X$, the formula $\eta_X(\mathsf{N})$ in \PIncz such that $\psi_X\equiv\Psi_X$. If $X=\emptyset$, then by definition, $\Psi_X=\bot$ and we can take $\eta_X=\bot$. Now, assume that $X\neq\emptyset$. First, consider the formula
\[\Phi_X:=\bigwedge_{v\in X}\underline{v(1)}\dots \underline{v(n)}\subseteq p_{1}\dots p_{n},\] 
where, once again, $v(i)$ is short for $v(p_i)$,  $\underline{0}:=\bot$ and $\underline{1}:=\top$. Each conjunct in $\Phi_X$ ensures (modulo the empty team) that the valuation $v$ must belong to the team $Y$ in question, in the sense that for any $\mathsf{N}$-team $Y$,
\begin{equation}\label{PT_exp_pw_eq0} 
Y\models\underline{v(1)}\dots \underline{v(n)}\subseteq p_{1}\dots p_{n}\iff Y=\emptyset\text{ or }v\in Y.
\end{equation}
To see why (\ref{PT_exp_pw_eq0}) holds, note that if $Y=\emptyset$, then $Y\models \underline{v(1)}\dots \underline{v(n)}\subseteq p_{1}\dots p_{n}$. If $Y\neq\emptyset$, then
$Y\models \underline{v(1)}\dots \underline{v(n)}\subseteq p_{1}\dots p_{n}$, iff there exists $u\in Y$ such that 
\[u(p_i)=u(\underline{v(i)})=v(p_i)\text{ for all }1\leq i\leq n,\]
iff $v\in Y$.
Now, by (\ref{PT_exp_pw_eq0}), we obtain that 
\begin{equation}\label{PT_exp_pw_eq1} 
Y\models\Phi_X\iff Y=\emptyset,\text{ or }v\in Y\text{ for all }v\in X\iff Y=\emptyset\text{ or }X\subseteq Y.
\end{equation}

Recall that we have defined a formula $\Theta_X$ in the proof of \Cref{cpl_excmp}, and the equivalence (\ref{cpl_excmp_eq2}) holds for the formula $\Theta_X$. Define now $\eta_X=\Theta_X\wedge\Phi_X$\footnote{This \PIncz-formula is essentially adapted from a very similar and slightly more complex modal formula in \cite{HellaStumpf15}, 
which uses  the more general extended inclusion atoms. 
}. By (\ref{cpl_excmp_eq2}) and (\ref{PT_exp_pw_eq1}), we obtain that for any $\mathsf{N}$-team $Y$, 
\[Y\models \psi_X\iff Y\subseteq X,\text{ and }Y=\emptyset\text{ or }X\subseteq Y\iff Y=\emptyset\text{ or }Y=X.\]
By (\ref{eq_psi_x_prop}) and locality, this implies that $\eta_X\equiv\Psi_X$, 
as we wanted.

\vspace{0.5\baselineskip}

Next, for every $\mathsf{N}$-team $X$, we define a formula $\psi_X$ in \Picnst  such that $\psi_X\equiv\Psi_X$. Let $\mathsf{N}=\{p_1,\dots,p_n\}$. We will define, inductively, for each subset $\mathsf{K}=\{p_{1},\dots,p_{k}\}$ ($k\leq n$) of propositional variables from $\mathsf{N}$ a \Picnst-formula $\psi_X^{\mathsf{K}}(\mathsf{K})$ such that $\psi_X^{\mathsf{K}}\equiv \Psi_X^{\mathsf{K}}$, where
\[\Psi_X^{\mathsf{K}}=\bigror_{v\in X}(p_{1}^{v(1)}\wedge\dots\wedge p_{k}^{v(k)}).\]
The above formula $\Psi_X^{\mathsf{K}}$ can be viewed as an approximation of the formula $\Psi_X$, and obviously $\Psi_X=\Psi_X^{\mathsf{N}}$. The required formula $\psi_X$ can thus be defined as $\psi_X=\psi_X^{\mathsf{N}}$. We now give the definition of $\psi_X^{\mathsf{K}}$ for every $\mathsf{N}$-team $X$ by induction on $|\mathsf{K}|$.

If $\mathsf{K}=\{p_1\}$, we have 
\[\displaystyle\Psi_X^{\mathsf{K}}=\big(\bigror_{v\in X^+} p_{1}\big)\ror\big(\bigror_{v\in X^-} \neg p_{1}\big),\] 
where $X^+=\{v\in X\mid v(p_1)=1\}$ and $X^-=\{v\in X\mid v(p_1)=0\}$.  If $X^+=\emptyset$, then $\Psi_X^{\mathsf{K}}=\bigroril_{v\in X^-} \neg p_{1}\equiv \neg p_{1}$, and we define $\psi_X^{\mathsf{K}}=\neg p_1$. If $X^-=\emptyset$, then $\Psi_X^{\mathsf{K}}\equiv  p_{1}$ and we define $\psi_X^{\mathsf{K}}=p_1$. If $X^+,X^-\neq\emptyset$, then $\displaystyle\Psi_X^{\mathsf{K}}=p_{1}\ror \neg p_{1}\equiv (p_1\sor \neg p_1)\wedge\anm(p_{1})$, and we define $\psi_X^{\mathsf{K}}=(p_1\sor \neg p_1)\wedge\anm(p_{1})$. 

If $\mathsf{K}=\{p_1,\dots,p_{m+1}\}=\mathsf{K}_0\cup\{p_{m+1}\}$, let $X^+=\{v\in X\mid v(p_{m+1})=1\}$ and $X^-=\{v\in X\mid v(p_{m+1})=0\}$. If $X^+=\emptyset$, then 
\begin{align*}
\Psi_X^{\mathsf{K}}&=\bigror_{v\in X^-} (p_{1}^{v(1)}\wedge\dots\wedge p_{m}^{v(m)}\wedge \neg p_{m+1})\\
&\equiv\big(\bigror_{v\in X^-} (p_{1}^{v(1)}\wedge\dots\wedge p_{m}^{v(m)})\big)\wedge \neg p_{m+1}\\
&\equiv\psi_{X^-}^{\mathsf{K}_0}\wedge \neg p_{m+1},\tag{by induction hypothesis}
\end{align*}
and we define $\psi_{X}^{\mathsf{K}}=\psi_{X^-}^{\mathsf{K}_0}\wedge \neg p_{m+1}$.
Similarly, if $X^-=\emptyset$, then $\Psi_X^{\mathsf{K}}\equiv \psi_{X^+}^{\mathsf{K}_0}\wedge p_{m+1}$, and we define $\psi_{X}^{\mathsf{K}}=\psi_{X^+}^{\mathsf{K}_0}\wedge  p_{m+1}$.
If $X^+,X^-\neq \emptyset$, by induction hypothesis we have  that 
\begin{align*}
\Psi_X^{\mathsf{K}} &\equiv(\psi_{X^+}^{\mathsf{K}_0}\wedge p_{m+1})\ror(\psi_{X^-}^{\mathsf{K}_0}\wedge \neg p_{m+1})
\equiv\big((\psi_{X^+}^{\mathsf{K}_0}\wedge p_{m+1})\sor(\psi_{X^-}^{\mathsf{K}_0}\wedge \neg p_{m+1})\big)\wedge \anm(p_{m+1}).
\end{align*}
Then, we define $\psi_{X}^{\mathsf{K}}=\big((\psi_{X^+}^{\mathsf{K}_0}\wedge p_{m+1})\sor(\psi_{X^-}^{\mathsf{K}_0}\wedge \neg p_{m+1})\big)\wedge \anm(p_{m+1})$.
\end{proof}

The above expressive completeness proof gives rise to normal forms of formulas in the logics \PIncz, \PInc,  \Picnst and \PAm.

\begin{corollary}[Normal form]\label{nf-pinc-picnst}
\begin{enumerate}[(i)]
\item\label{nf-pinc-picnst-1} Every formula $\phi(\mathsf{N})$ in \PIncz and \PInc is equivalent to a formula of the form $\bigvee_{X\in \mathcal{X}}(\Theta_X\wedge\Phi_X)$  for some collection $\mathcal{X}$ of $\mathsf{N}$-teams.
\item\label{nf-pinc-picnst-2} Every formula $\phi(\mathsf{N})$ in \Picnst and \PAm is equivalent to a formula of the form $\bigvee_{X\in \mathcal{X}}\psi_X$  for some collection $\mathcal{X}$ of $\mathsf{N}$-teams.
\end{enumerate}
\end{corollary}

Let us end this section with some further discussion on the normal forms of the logics \PInc and \PAm, or the normal forms of inclusion atoms and anonymity atoms in the logics in particular.
First, note that  \PInc-formulas in the normal form $\bigvee_{X\in \mathcal{X}}(\Theta_X\wedge\Phi_X)$ contain primitive inclusion atoms $\mathsf{x}\subseteq \mathsf{p}$ only 
Similarly, \PAm-formulas in the normal form $\bigvee_{X\in \mathcal{X}}\psi_X$ contain inconstancy atoms $\idep(p)$ with single arguments only. It then follows that arbitrary (nontrivial) inclusion atoms $\mathsf{b}\subseteq\mathsf{c}$ and anonymity atoms $\mathsf{p}\Upsilon\mathsf{q}$ are definable in terms of regular primitive inclusion atoms $\mathsf{x}\subseteq \mathsf{p}$  and inconstancy atoms $\idep(p)$ with single arguments, respectively. 
We provide  direct definitions for these atoms in terms of the corresponding simpler atoms in the following. 



To simplify notations,  we write $v(i)$ for $v(p_i)$. The notation $p^1$ or $p^{\top}$ stand for $p$, and $p^0$ or $p^{\bot}$ stand for $\neg p$; similarly, $\top^\bot$ and $\bot^{\top}$ both stand for $\bot$. Hereafter, we reserve the letters $p,q,r,\dots$ with or without subscripts for propositional variables, the letters $x,y,z,\dots$ with or without subscripts for  the constants $\bot$ and $\top$, and the letters $a,b,c,\dots$ with or without subscripts for either propositional variables or $\bot$ or $\top$. The serif font letters  $\mathsf{p},\mathsf{q},\dots$ will stand for sequences of propositional variables of certain lengths; similarly for $\mathsf{x},\mathsf{y},\dots$ in serif font, and $\mathsf{a},\mathsf{b},\mathsf{c},\dots$ in serif font. We write $|\mathsf{a}|$ for the length of the sequence $\mathsf{a}$. For two sequences $\mathsf{a}=\langle a_1,\dots, a_n\rangle$ and $\mathsf{x}=\langle x_1,\dots,x_n\rangle$ with each $a_i\in \textsf{Prop}\cup\{\top,\bot\}$ and  $x_i\in\{\top,\bot\}$, we write $\mathsf{a}^{\mathsf{x}}$ for $a_1^{x_1}\wedge \dots\wedge a_n^{x_n}$. 


\begin{proposition}\label{anonymity2ror_df}
\begin{enumerate}[(i)]
\item $\mathsf{p}\Upsilon q_1\dots q_m \equiv \mathsf{p}\Upsilon q_1\vee\dots\vee \mathsf{p}\Upsilon q_m\text{ and }\mathsf{p}\Upsilon\langle\rangle\equiv\bot$,
\item $\displaystyle p_1\dots p_k\Upsilon q \equiv  \bigvee_{v\in 2^{\mathsf{K}}}(p_1^{v(1)} \wedge\dots\wedge p_k^{v(k)}\wedge \idep(q)),\text{ where }\mathsf{K}=\{p_1,\dots,p_k\},$
\item $\mathsf{p}\Upsilon \langle\rangle\vdash\bot$
\end{enumerate}
\end{proposition}
\begin{proof}
Easy.
\end{proof}

\begin{proposition}\label{inc_2prim}
\begin{enumerate}[(i)]
\item\label{inc_2prim_1}
$\bigwedge_{\mathsf{x}\in\{\top,\bot\}^{|\mathsf{a}|}}\big(\mathsf{a}^{\mathsf{x}}\to\mathsf{x}\subseteq \mathsf{b}\big)\equiv\mathsf{a}\subseteq \mathsf{b}$.
\item\label{inc_2prim_3} If $|\mathsf{a}|=|\mathsf{c}|$ and $|\mathsf{b}|=|\mathsf{d}|$, then $\mathsf{a}\subseteq \mathsf{a}\equiv\top$,
\[\mathsf{a}\top\mathsf{b}\subseteq \mathsf{c}\top\mathsf{d}\equiv \mathsf{a}\mathsf{b}\subseteq \mathsf{c}\mathsf{d}\equiv\mathsf{a}\bot\mathsf{b}\subseteq \mathsf{c}\bot\mathsf{d}\text{ and } \mathsf{a}\top\mathsf{b}\subseteq \mathsf{c}\bot\mathsf{d}\equiv \bot\equiv\mathsf{a}\bot\mathsf{b}\subseteq \mathsf{c}\top\mathsf{d}.\]
\end{enumerate}
\end{proposition}
\begin{proof}
Item (\ref{inc_2prim_3}) is easy to prove. We only give the detailed proof for item (\ref{inc_2prim_1}). For the left to right direction, suppose $X\models \neg\mathsf{a}^{\mathsf{x}}\vee\mathsf{x}\subseteq \mathsf{b}$ for all $\mathsf{x}\in\{\top,\bot\}^{|\mathsf{a}|}$. Let $v\in X$ and let $\mathsf{x}\in\{\top,\bot\}^{|\mathsf{a}|}$ be such that $v(\mathsf{x})=v(\mathsf{a})$. By assumption, there exist $Y,Z\subseteq X$ such that $X=Y\cup Z$, $Y\models \neg\mathsf{a}^{\mathsf{x}}$ and $Z\models \mathsf{x}\subseteq \mathsf{b}$. Clearly, $v\notin Y$, which means $v\in Z$. Then, there exists $u\in Z$ such that $u(\mathsf{b})=v(\mathsf{x})=v(\mathsf{a})$. Hence, $X\models \mathsf{a}\subseteq \mathsf{b}$.

Conversely, suppose $X\models \mathsf{a}\subseteq \mathsf{b}$. We show that $X\models \mathsf{a}^{\mathsf{x}}\to\mathsf{x}\subseteq \mathsf{b}$ for any $\mathsf{x}\in\{\top,\bot\}^{|\mathsf{a}|}$. Let $X_{\mathsf{x}}=\{v\in X\mid v(\mathsf{a})\neq v(\mathsf{x})\}$. Clearly, $X_{\mathsf{x}}\models \neg \mathsf{a}^{\mathsf{x}}$. If $X_{\mathsf{x}}=X$, then  $X\models \neg \mathsf{a}^{\mathsf{x}}$ and thus $X\models \neg \mathsf{a}^{\mathsf{x}}\vee \mathsf{x}\subseteq \mathsf{b}$ as required. Otherwise, $X\setminus X_{\mathsf{x}}\neq\emptyset$. 
We show that $X\models \mathsf{x}\subseteq \mathsf{b}$, which would suffice. Let $u\in X$ and pick any $v\in X\setminus X_{\mathsf{x}}$. We have that $u(\mathsf{x})=v(\mathsf{x})=v(\mathsf{a})$. Since $X\models \mathsf{a}\subseteq \mathsf{b}$, there exists $w\in X$ such that $w(\mathsf{b})=v(\mathsf{a})=u(\mathsf{x})$, as required.
\end{proof}

Moreover, we show that 
primitive inclusion atoms can actually be defined in terms of such atoms of arity $1$, i.e., inclusion atoms of the form $\top\subseteq p$ or $\bot\subseteq q$.

\begin{proposition}\label{inc_2prim_2}
$\mathsf{x}y\subseteq\mathsf{p}q\equiv \mathsf{x}\subseteq\mathsf{p}\wedge \big((y\subseteq q\wedge \mathsf{p}^{\mathsf{x}})\vee\neg \mathsf{p}^{\mathsf{x}}\big)$.
\end{proposition}
\begin{proof}
For the  left to right direction, suppose $X\models \mathsf{x}y\subseteq\mathsf{p}q$. Then clearly $X\models \mathsf{x}\subseteq\mathsf{p}$ as well. It remains to show $X\models(y\subseteq q\wedge \mathsf{p}^{\mathsf{x}})\vee\neg \mathsf{p}^{\mathsf{x}}$. Define 
\[Y=\{v\in X\mid v(\mathsf{p})=v(\mathsf{x})\}\]
and $Z=X\setminus Z$. Clearly, $Y\cup Z=X$, $Y\models \mathsf{p}^{\mathsf{x}}$ and $Z\models \neg \mathsf{p}^{\mathsf{x}}$. We now show that $Y\models y\subseteq q$. For any $v\in Y$, we have $v(\mathsf{p})=v(\mathsf{x})$ by definition. Also, since $Y\subseteq X\models \mathsf{x}y\subseteq\mathsf{p}q$, there exists $u\in X$ such that $v(\mathsf{x})=u(\mathsf{p})$ and $v(y)=u(q)$. From $u(\mathsf{p})=v(\mathsf{x})=u(\mathsf{x})$, we conclude that $u\in Y$. Hence, $Y\models y\subseteq q$.

Conversely, suppose $X\models \mathsf{x}\subseteq\mathsf{p}\wedge \big((y\subseteq q\wedge \mathsf{p}^{\mathsf{x}})\vee\neg \mathsf{p}^{\mathsf{x}}\big)$. For any $v\in X$, since $X\models \mathsf{x}\subseteq\mathsf{p}$, there exists $u\in X$ such that $u(\mathsf{p})=v(\mathsf{x})=u(\mathsf{x})$. On the other hand, since $X\models (y\subseteq q\wedge \mathsf{p}^{\mathsf{x}})\vee\neg \mathsf{p}^{\mathsf{x}}$, there exist $Y,Z\subseteq X$ such that $X=Y\cup Z$, $Y\models y\subseteq q\wedge \mathsf{p}^{\mathsf{x}}$ and $Z\models \neg \mathsf{p}^{\mathsf{x}}$. Clearly, $u\notin Z$ and so $u\in Y\neq\emptyset$. Now, since $Y\models y\subseteq q$, there exists $w\in Y$ such that $w(q)=u(y)=v(y)$. From $Y\models \mathsf{p}^{\mathsf{x}}$, we also conclude that $w(\mathsf{p})=w(\mathsf{x})=v(\mathsf{x})$. Hence $X\models \mathsf{x}y\subseteq\mathsf{p}q$.
\end{proof}

Putting these results together, we obtain that an arbitrary nontrivial inclusion atom $\mathsf{a}\subseteq\mathsf{b}$ can be decomposed into very simple inclusion atoms of the form $\bot\subseteq p$ and $\top\subseteq p$.

\begin{corollary}\label{inc_atm2_reg_primitive_atm_sem}
An arbitrary inclusion atom $\mathsf{a}\subseteq\mathsf{b}$ is either equivalent to $\bot$ or $\top$, or it can be expressed in terms of  primitive inclusion atoms of arity $1$.
\end{corollary}
\begin{proof}
Given an arbitrary inclusion atom $\mathsf{a}\subseteq\mathsf{b}$, by applying \Cref{inc_2prim}(\ref{inc_2prim_1}), we obtain an equivalent formula $\phi$ in which all inclusion atoms are primitive and of the form $\mathsf{x}\subseteq \mathsf{b}$. Next,  apply \Cref{inc_2prim}(\ref{inc_2prim_3}) to remove the constants $\bot$ and $\top$ from the sequence $\mathsf{b}$ on the right-hand side of all inclusion atoms $\mathsf{x}\subseteq \mathsf{b}$. Some of the resulting formulas are equivalent to $\bot$ or $\top$. Finally, we apply \Cref{inc_2prim_2} exhaustedly to turn every primitive inclusion atom $\mathsf{x}\subseteq\mathsf{p}$ obtained in the previous step into an equivalent formula in which inclusion atoms are all of arity $1$.
\end{proof}

It is interesting to note that \Cref{inc_atm2_reg_primitive_atm_sem} also gives rise to a direct definition of inclusion atoms $\mathsf{a}\subseteq\mathsf{b}$ in terms of relevant disjunction $\ror$: First transform $\mathsf{a}\subseteq\mathsf{b}$ to a formula that contains primitive inclusion atoms $\bot\subseteq p$ or $\top\subseteq p$ of arity $1$. Then observe that
\(\bot\subseteq p\equiv \neg p\ror \top\text{ and }\top\subseteq p\equiv p\ror\top.\)

\section{Axiomatizations}

In this section, we axiomatize the union closed team logics \PUw and \PInc as well as \PIncz. We define systems of natural deduction for these logics and prove the completeness theorem for these systems. Our argument for the completeness proof makes heavy and essential use of the disjunctive normal form of the logics given in  \Cref{nf-pu} and \Cref{nf-pinc-picnst} from the previous section. Such a technique is a generalization of the similar ones developed in  \cite{VY_PD,YangVaananen:17PT} for propositional team logics. As seen in \Cref{nf-pinc-picnst}(ii), the disjunctive normal form for the logic \PAm is substantially more complex, 
our approach thus does not suit  well for \PAm. We leave the axiomatization for \PAm for future work.

\subsection{\PUw}\label{sec:pu}

In this subsection, we define a system of natural deduction for \PUw and prove the completeness theorem.

Let us first present the system. We adopt the standard conventions of systems of natural deduction; readers who are not familiar with natural deduction systems are referred to, e.g., \cite{TroelstraSchwichtenberg96,logic_structure_vanDalen}. For example, the letter $D$ (with or without subscripts) in the following definition stands for an arbitrary derivation. 

\begin{table}[t]
 \begin{center}
\caption{Rules for constants and classical connectives
}
\setlength{\tabcolsep}{6pt}
\renewcommand{\arraystretch}{1.8}
\setlength{\extrarowheight}{1pt}
\vspace{2pt}
\scalebox{.96}{\begin{tabular}{|C{0.4\linewidth}C{0.52\linewidth}|}
\hline
\multicolumn{2}{|c|}{\AxiomC{}\RightLabel{\topi}\UnaryInfC{$\top$} \DisplayProof\quad
\AxiomC{}\noLine\UnaryInfC{[$\alpha$]}\noLine\UnaryInfC{{\small$D$}}\noLine\UnaryInfC{$\bot$} \RightLabel{$\neg$\textsf{I} {\footnotesize (1)}}\UnaryInfC{$\neg\alpha$}\noLine\UnaryInfC{}\noLine\UnaryInfC{}\DisplayProof
\quad\AxiomC{{\small$D_0$}}\noLine\UnaryInfC{$\alpha$}\AxiomC{{\small$D_1$}}\noLine\UnaryInfC{$\neg\alpha$}\RightLabel{$\neg$\textsf{E}}\BinaryInfC{$\phi$} \DisplayProof
\quad\quad\AxiomC{[$\neg\alpha$]}\noLine\UnaryInfC{{\small$D$}}\noLine\UnaryInfC{$\bot$}\RightLabel{\textsf{RAA} {\footnotesize (1)}}\UnaryInfC{$\alpha$}\noLine\UnaryInfC{} \DisplayProof 
}\\
\AxiomC{{\small$D_0$}}\noLine\UnaryInfC{$\phi$}\AxiomC{{\small$D_1$}}\noLine\UnaryInfC{$\psi$}\RightLabel{\ci}\BinaryInfC{$\phi\wedge\psi$} \DisplayProof &
\AxiomC{{\small$D$}}\noLine\UnaryInfC{$\phi\wedge\psi$} \RightLabel{\ce}\UnaryInfC{$\phi$} \DisplayProof\quad\AxiomC{{\small$D$}}\noLine\UnaryInfC{$\phi\wedge\psi$} \RightLabel{\ce}\UnaryInfC{$\psi$} \DisplayProof\\
\multirow{2}{*}{\AxiomC{{\small$D$}}\noLine\UnaryInfC{$\phi$}\RightLabel{\sori} \UnaryInfC{$\phi\sor\psi$} \DisplayProof\AxiomC{{\small$D$}}\noLine\UnaryInfC{$\phi$}\RightLabel{\sori} \UnaryInfC{$\psi\sor\phi$}  \DisplayProof}
&\AxiomC{{\small$D$}}\noLine\UnaryInfC{$\phi\vee\psi$} \AxiomC{[$\phi$]}\noLine\UnaryInfC{{\small$D_0$}}\noLine\UnaryInfC{$\chi$} \AxiomC{}\noLine\UnaryInfC{[$\psi$]}\noLine\UnaryInfC{{\small$D_1$}}\noLine\UnaryInfC{$\chi$} \RightLabel{$\vee\textsf{E}$ {\footnotesize (2)}}\TrinaryInfC{$\chi$}\DisplayProof\\[-4pt]
\multicolumn{2}{|l|}{\footnotesize{(1) The undischarged assumptions\footnotemark in the derivation $D$ contain classical formulas only.}} \\[-14pt]
\multicolumn{2}{|l|}{\footnotesize{(2) The undischarged assumptions in the derivations $D_0$ and $D_1$ contain classical formulas only.}} \\\hline
\end{tabular}
}
\label{tab_rule_general}
\vspace{-28pt}
\end{center}
\end{table}
\begin{table}[H]
 \begin{center}
\caption{Rules for $\ror$ and interactions}
\vspace{2pt}
\setlength{\tabcolsep}{6pt}
\renewcommand{\arraystretch}{1.8}
\setlength{\extrarowheight}{1pt}
\scalebox{.96}{\begin{tabular}{|C{0.48\linewidth}C{0.45\linewidth}|}\hline
\AxiomC{}\noLine\UnaryInfC{{\small$D_0$}}\noLine\UnaryInfC{$\phi$}\AxiomC{}\noLine\UnaryInfC{{\small$D_1$}}\noLine\UnaryInfC{$\psi$} \RightLabel{$\ror$\textsf{I}}\BinaryInfC{$\phi\ror\psi$} \noLine\UnaryInfC{}\DisplayProof&\multirow{2}{*}{\AxiomC{{\small$D_0$}}\noLine\UnaryInfC{$\phi\ror\psi$} \AxiomC{}\noLine\UnaryInfC{$[\phi]$}\noLine\UnaryInfC{{\small$D_1$}}\noLine\UnaryInfC{$\chi$} \RightLabel{$\ror\textsf{Mon}$ {\footnotesize(1)}}\BinaryInfC{$\chi\ror\psi$}\DisplayProof} 
 \\
\AxiomC{{\small$D$}}\noLine\UnaryInfC{$\phi\ror\psi$} \RightLabel{$\ror$\textsf{Com}}\UnaryInfC{$\psi\ror\phi$}\noLine\UnaryInfC{}\DisplayProof\AxiomC{{\small$D$}}\noLine\UnaryInfC{$\phi\ror(\psi\ror\chi)$} \RightLabel{$\ror$\textsf{Ass}}\UnaryInfC{$(\phi\ror\psi)\ror\chi$}\noLine\UnaryInfC{}\DisplayProof
&
\\
\AxiomC{{\small$D$}}\noLine\UnaryInfC{$\phi\sor\psi$} \AxiomC{}\noLine\UnaryInfC{[$\phi$]}\noLine\UnaryInfC{{\small$D_0$}}\noLine\UnaryInfC{$\chi$} 
\AxiomC{}\noLine\UnaryInfC{[$\psi$]}\noLine\UnaryInfC{{\small$D_1$}}\noLine\UnaryInfC{$\chi$} 
\AxiomC{}\noLine\UnaryInfC{$[\phi\ror\psi]$}\noLine\UnaryInfC{{\small$D_2$}}\noLine\UnaryInfC{$\chi$} \RightLabel{$\sor\textsf{E}_{\ror}$}\QuaternaryInfC{$\chi$}\noLine\UnaryInfC{}\DisplayProof &\AxiomC{{\small$D$}}\noLine\UnaryInfC{$\phi\ror\psi$}\RightLabel{$\ror\sor$\textsf{Tr}} \UnaryInfC{$\phi\sor\psi$}  \DisplayProof\\
\AxiomC{{\small$D$}}\noLine\UnaryInfC{$\phi\ror\bot$}\RightLabel{$\ror\bot$\textsf{E}}\UnaryInfC{$\psi$} \DisplayProof
&\AxiomC{{\small$D$}}\noLine\UnaryInfC{$\phi\ror(\psi\sor\chi)$} \RightLabel{$\dstr\ror\sor$}\UnaryInfC{$(\phi\ror\psi)\sor(\phi\ror\chi)$} \DisplayProof\\[-4pt]
\multicolumn{2}{|l|}{\footnotesize  (1) The undischarged assumptions in the derivation $D_1$ contain classical formulas only.} \\\hline
\end{tabular}}
\label{tab_rule_ror}
\end{center}
\end{table}
\vspace{-26pt}

\footnotetext{When this rule is applied, the (open) assumption $\alpha$ at the top of the branch of the derivation will be deemed as  {\em closed} and thus discharged (from the set of open assumptions of the derivation). All the remaining assumptions in the derivation are regarded as ``undischarged'' assumptions.}

\begin{definition}
The  system of \PU consists of all rules  given in \Cref{tab_rule_general,tab_rule_ror}, where $\alpha$ ranges over classical formulas only.

We write $\Gamma\vdash_{\PU}\phi$ or simply $\Gamma\vdash\phi$ if $\phi$ is derivable from the set  $\Gamma$ of formulas by applying the rules of the system of \PU. We write simply $\phi\vdash\psi$ for $\{\phi\}\vdash\psi$. Two formulas $\phi$ and $\psi$ are said to be {\em provably equivalent}, written $\phi\dashv\vdash\psi$, if both $\phi\vdash\psi$ and $\psi\vdash\phi$. 
\end{definition}

Our system does not admit uniform substitution, as, e.g., the rules for negation $\neg$ apply to classical formulas  only. When restricted to classical formulas the system coincides with the system of classical propositional logic. In particular, the disjunction $\vee$ and the negation $\neg$ admit the usual elimination rule $\vee\textsf{E}$, introduction rule $\neg$\textsf{I}  and {\em reductio ad absurdum} rule \textsf{RAA}, respectively, under the condition that the undischarged assumptions in the derivations involved contain classical 
formulas only. 
It is interesting to note that the soundness of  the disjunction elimination rule $\vee\textsf{E}$  is a nontrivial  feature of the union closed team logics, especially because this same rule is actually not sound for the propositional team logics with the downwards closure property or without any closure property (see \cite{VY_PD,YangVaananen:17PT}).

The rules for the relevant disjunction $\ror$ are peculiar.
 Unsurprisingly, the usual introduction rule ($\phi/\phi\ror\psi$) is not sound for the relevant disjunction $\ror$, because, e.g., obviously $\phi\not\models\phi\ror\bot$. While the relevant disjunction introduction rule $\ror$\textsf{I} we have in the system  is considerably weak,
%
%
%
%
the relevant disjunction $\ror$ does admit the usual elimination rule under the same side condition as that for $\vee\textsf{E}$. We will show in the next proposition that such restricted elimination rule is derivable from monotonicity rule of the relevant disjunction $\ror\textsf{Mon}$. The rules $\ror\textsf{Com}$ and $\ror\textsf{Ass}$ are added in the system also in order to compensate the weakness of the nonstandard introduction and elimination rule for $\ror$.
The two rules $\sor\textsf{E}_{\ror}$ and $\ror\sor\textsf{Tr}$ together simulate the evident equivalence 
\(\phi\sor\psi\equiv\phi\ior\psi\ior(\phi\ror\psi)\) with $\sor\textsf{E}_{\ror}$ simulating  the left to right direction and $\ror\sor\textsf{Tr}$ simulating the right to left direction.
The rule $\ror\bot$\textsf{E} characterizes the fact that each disjunct in a relevant disjunction has to be satisfied by a nonempty team (if the starting team is not empty).
The distributive rule $\dstr\ror\sor$ is actually invertible, as we will show in the next proposition that lists also some other useful clauses for our system.

\begin{proposition}\label{der_rule_prop}
\begin{enumerate}[(i)]
\item\label{der_rule_dstr_rore} Let $\Delta$ be a set of classical formulas. If $\Delta,\phi\vdash\chi$ and $\Delta,\psi\vdash\chi$, then $\Delta,\phi\ror\psi\vdash\chi$. 
\item\label{der_rule_dstr_roridempotent} $\phi\ror\phi\dashv\vdash\phi$.
\item\label{der_rule_dstr_rorsor} $\phi\ror(\psi\sor\chi)\dashv\vdash(\phi\ror\psi)\sor(\phi\ror\chi)$.
\item\label{der_rule_exfalso} $\bot\vdash\phi$.
\item\label{der_rule_bote} $\phi\ror\alpha,\neg\alpha\vdash\psi$

\end{enumerate}
\end{proposition}
\begin{proof}
For item (\ref{der_rule_dstr_rore}), since $\Delta,\phi\vdash\chi$, we derive by $\ror\textsf{Sub}$ that $\Delta,\phi\ror\psi\vdash\chi\ror\psi$. Similarly, from $\Delta,\psi\vdash\chi$ we derive $\Delta,\chi\ror\psi\vdash\chi\ror\chi$. Thus, $\Delta,\phi\ror\psi\vdash\chi\ror\chi$. By $\ror\sor$\textsf{Tr} and $\vee\textsf{E}$ we derive $\chi\ror\chi\vdash\chi\vee\chi\vdash\chi$. Hence we conclude $\Delta,\phi\ror\psi\vdash\chi$.

For item (\ref{der_rule_dstr_roridempotent}), the left to right direction is a special case of item (\ref{der_rule_dstr_rore}), and the right to left direction follows from $\ror\textsf{I}$.

For item (\ref{der_rule_dstr_rorsor}), the left to right direction follows from $\dstr\ror\sor$. For the other direction, by $\vee\textsf{E}$ it suffices to prove $\phi\ror\psi\vdash \phi\ror(\psi\sor\chi)$ and $\phi\ror\chi\vdash\phi\ror(\psi\sor\chi)$. But these follow easily from $\vee\textsf{I}$ and $\ror\textsf{Mon}$.


Item (\ref{der_rule_exfalso}) is proved by the usual argument by applying $\neg\textsf{I}$ and $\neg\textsf{E}$.

For item (\ref{der_rule_bote}), by $\neg$\textsf{E}, we have $\alpha,\neg\alpha\vdash \bot$, which gives $\phi\ror\alpha,\neg\alpha\vdash\phi\ror\bot$ by $\ror\textsf{Mon}$. Furthermore, we have $\phi\ror\bot\vdash\psi$ by $\ror\bot$\textsf{E}. Hence $\phi\ror\alpha,\neg\alpha\vdash\psi$.
%
%
\end{proof}

\begin{theorem}[Soundness]
For any set $\Gamma\cup\{\phi\}$ of \PUw-formulas, we have that $\Gamma\vdash\phi\Longrightarrow\Gamma\models\phi$.
\end{theorem}
\begin{proof}
The soundness of the rules in \Cref{tab_rule_general} and the first four rules in \Cref{tab_rule_ror} are easy to verify. 
The soundness of $\sor\textsf{E}_{\ror}$ and $\ror\sor$\textsf{Tr} follow from the fact that $X\models\phi\sor\psi$ if and only if $X\models\phi$ or $X\models\psi$ or $X\models\phi\ror\psi$. The rule $\ror\bot$\textsf{E} is also clearly sound, since the assumption $\phi\ror\bot$ is satisfied only by the empty team, which satisfies every formula $\psi$. We only verify the soundness of the  rule $\dstr\ror\sor$.

Assuming that $X\models \phi\ror(\psi\sor\chi)$ for some nonempty team $X$ we show that $X\models(\phi\ror\psi)\sor(\phi\ror\chi)$. By the assumption, there are nonempty teams $Y,Z\subseteq X$ such that $X=Y\cup Z$, $Y\models\phi$ and $Z\models \psi\vee\chi$. The latter implies that there are subteams $W,U\subseteq Z$ such that $Z=W\cup U$, $W\models\psi$ and $U\models \chi$. If $W=\emptyset$, then $U\neq\emptyset$ as $Z\neq \emptyset$. In this case $X=Y\cup U\models\phi\ror\chi$ and thus $X\models (\phi\ror\psi)\sor(\phi\ror\chi)$. Symmetrically, if $U=\emptyset$, then $W\neq\emptyset$ and $X\models (\phi\ror\psi)\sor(\phi\ror\chi)$ as well. Lastly, if $W,U\neq \emptyset$, then $Y\cup W\models\phi\ror\psi$ and $Y\cup U\models\phi\ror\chi$. Thus, we have that $(Y\cup W)\cup (Y\cup U)=X\models(\phi\ror\psi)\sor(\phi\ror\chi)$.
%
%
\end{proof}

The rest of this section is devoted to the proof of the completeness theorem of our system. We will show that every \PU-formula is provably equivalent to a formula in the normal form $\bigvee_{X\in \mathcal{X}}\Psi_X$ given by \Cref{nf-pu}.

\begin{lemma}\label{DNF_PU}
Let $\mathsf{N}=\{p_1,\dots,p_n\}$. Every \PUw-formula $\phi(\mathsf{N})$  is provably equivalent to a formula of the  form
\begin{equation}\label{NF_CTLe_eq}
\bigsor_{X\in\mathcal{X}}\Psi_{X},~\text{ where }~\Psi_{X}=\bigror_{v\in X}(p_{1}^{v(1)}\wedge\dots\wedge p_{n}^{v(n)}),
\end{equation}
 and $\mathcal{X}$ is a  finite set of $\mathsf{N}$-teams.
\end{lemma}

We shall postpone the technical proof of this above lemma till the end of this section.
The completeness then follows from some derivations in the system that uses the specific syntactic shape of the normal form. One important step in this proof is to obtain from the semantic side that the entailment $\bigsor_{X\in\mathcal{X}}\Psi_{X}\models\bigsor_{Y\in\mathcal{Y}}\Psi_{Y}$ of two formulas in the disjunctive normal form implies that each team $X$ in $\mathcal{X}$ is identical to the union of all teams from a subcollection of $\mathcal{Y}$. We now prove this last semantic property and also its converse direction.



\begin{lemma}\label{comp_main_lm}
For any nonempty finite sets $\mathcal{X}$ and $\mathcal{Y}$ of $\mathsf{N}$-teams, the following are equivalent:
\begin{enumerate}[(i)]
\item $\displaystyle\bigsor_{X\in\mathcal{X}}\Psi_{X}\models\bigsor_{Y\in\mathcal{Y}}\Psi_{Y}$.
\item For each $X\in\mathcal{X}$, there exists $\mathcal{Y}_X\subseteq \mathcal{Y}$ such that $X= \bigcup\mathcal{Y}_X$.
\end{enumerate}
\end{lemma}
\begin{proof}
(i)$\Longrightarrow$(ii): For each $X_0\in \mathcal{X}$, we have  $X_0\models\Psi_{X_{0}}$ by  Equation (\ref{eq_psi_x_prop}). Thus $X_{0}\models\bigsor_{X\in\mathcal{X}}\Psi_{X}$,
which by (i) implies that $ X_{0}\models \bigsor_{Y\in\mathcal{Y}}\Psi_{Y}$.
This means that for each $Y\in\mathcal{Y}$, there exists $Z_Y\subseteq X_{0}$ such that  $X_{0}=\bigcup_{Y\in\mathcal{Y}}Z_Y$ and each $Z_Y\models \Psi_{Y}$. The latter implies, by Equation (\ref{eq_psi_x_prop}) again, that $Z_Y=Y$ or $Z_Y=\emptyset$. Thus  we obtain $ X_{0}=\bigcup_{Y\in\mathcal{Y}_{X}}Y$ for some $\mathcal{Y}_{X}\subseteq \mathcal{Y}$.

(ii)$\Longrightarrow$(i): Suppose $Z$ is any $\mathsf{N}$-team  satisfying $Z\models\bigsor_{X\in\mathcal{X}}\Psi_{X}$. Then, by Equation (\ref{eq_psi_x_prop}), there exists $\mathcal{X}'\subseteq \mathcal{X}$ such that $Z=\bigcup_{X\in\mathcal{X}'}X$. By (ii), for each $X\in\mathcal{X}'$, there exists $\mathcal{Y}_X\subseteq \mathcal{Y}$ such that $X=\bigcup\mathcal{Y}_X$. Thus, we have that $Z=\bigcup_{X\in\mathcal{X}'}\bigcup\mathcal{Y}_X=\bigcup\mathcal{Y}'$, where $\mathcal{Y}'=\bigcup_{X\in\mathcal{X}'}\mathcal{Y}_X\subseteq \mathcal{Y}$. Hence, $Z\models \bigsor_{Y\in\mathcal{Y'}}\Psi_{Y}$ by Eqaution (\ref{eq_psi_x_prop}) again,  thereby $Z\models \bigsor_{Y\in\mathcal{Y}}\Psi_{Y}$.
\end{proof}

Next, we prove a technical lemma that concerns an interesting interaction between the two disjunctions $\vee$ and $\ror$.

\begin{lemma}\label{lem_rorsor_trans}
Let $\mathcal{Y}$ be a finite set with each $Y\in \mathcal{Y}$ being a finite set of indices. 
Then $\bigroril_{i\in \bigcup\mathcal{Y}}\phi_i\vdash\bigvee_{Y\in\mathcal{Y}}\bigroril_{i\in Y}\phi_i$.
\end{lemma}
\begin{proof}
Note that elements in $\mathcal{Y}$ are not necessarily disjoint. We thus first derive by applying \rori, $\ror$\textsf{Ass} and $\ror$\textsf{Com} that $\bigroril_{i\in \bigcup\mathcal{Y}}\phi_i\vdash\bigroril_{Y\in\mathcal{Y}}\bigroril_{i\in Y}\phi_i$. Next, by repeatedly applying $\ror\sor$\textsf{Tr} and \sore, we derive that $\bigroril_{Y\in\mathcal{Y}}\bigroril_{i\in Y}\phi_i\vdash \bigvee_{Y\in\mathcal{Y}}\bigroril_{i\in Y}\phi_i$, which then implies the desired clause.
\end{proof}

Now, we give the  proof of the completeness theorem of our system.

\begin{theorem}[Completeness]\label{PU_completeness}
For any set $\Gamma\cup\{\phi\}$  of \PUw-formulas, we have that $\Gamma\models\phi \iff \Gamma\vdash\phi$. 
\end{theorem}
\begin{proof}
It suffices to prove the left to right direction. Suppose $\Gamma\models\phi$.
By the compactness theorem (\Cref{compactness_thm}) we may assume that $\Gamma$ is a finite set. Let $\psi=\bigwedge\Gamma$, and  suppose $\phi,\psi$ are formulas in $\mathsf{N}=\{p_1,\dots,p_n\}$. By  \Cref{DNF_PU},  
\begin{equation*}
\psi\dashv\vdash \bigsor_{X\in \mathcal{X}}\Psi_{X}\quad\text{ and }\quad\phi\dashv\vdash \bigsor_{Y\in\mathcal{Y}}\Psi_{Y}
\end{equation*}
 for some finite sets $\mathcal{X}$ and $\mathcal{Y}$ of $\mathsf{N}$-teams. The soundness theorem implies that
\begin{equation}\label{pu_compl_eq2}
\bigsor_{X\in \mathcal{X}}\Psi_{X}\models\bigsor_{Y\in \mathcal{Y}}\Psi_{Y}.
\end{equation}

If $\mathcal{X}=\emptyset$, then $\psi\dashv\vdash\bot$, and we derive  $\psi\vdash\phi$ by \Cref{der_rule_prop}(\ref{der_rule_exfalso}). If $\mathcal{Y}=\emptyset$, then $\phi\dashv\vdash\bot$. In view of (\ref{pu_compl_eq2}), it must be that $\mathcal{X}=\emptyset$ as well. Thus $\psi\dashv\vdash\bot$ giving that $\psi\vdash \phi$.

If $\mathcal{X},\mathcal{Y}\neq \emptyset$, then by Lemma \ref{comp_main_lm},  for each $X\in \mathcal{X}$ we have that $X=\bigcup\mathcal{Y}_X$ for some $\mathcal{Y}_X\subseteq\mathcal{Y}$.  
Thus, we derive $\Psi_{X}\vdash \bigsor_{Y\in \mathcal{Y}_X}\Psi_{Y}\vdash \bigsor_{Y\in \mathcal{Y}}\Psi_{Y}$ by \Cref{lem_rorsor_trans} and \sori.
Finally, we obtain
\(
\bigsor_{X\in \mathcal{X}}\Psi_{X}\vdash\bigsor_{Y\in \mathcal{Y}}\Psi_{Y}
\) by $\sor$\textsf{E},
thereby $\psi\vdash\phi$.
\end{proof}

Before we supply the proof of \Cref{DNF_PU}, let us first give an example of the  applications of our system of \PUw, in the context of the implication problem of anonymity atoms (or afunctional dependencies). 
Anonymity atoms $\mathsf{p}\Upsilon\mathsf{q}$ are definable in \PU, and recall from \Cref{anonymity2ror_df} the concrete definitions.
The implication problem of anonymity atoms (i.e., the problem of whether $\Gamma\models\phi$ for  a set $\Gamma\cup\{\phi\}$ of anonymity atoms) is shown in \cite{Vaananen_anonymity19} to be completely axiomatized by the rules listed in the next example (read the clauses in the example as rules). We now show that these rules are derivable in the system of \PUw (via the translation given in \Cref{anonymity2ror_df}).

\begin{example}
Let $\mathsf{p},\mathsf{q},\mathsf{r},\mathsf{s},\mathsf{p}',\mathsf{q}',\mathsf{r}'$ be sequences of propositional variables.
\begin{description}
\item[(i)] $\mathsf{pqr}\Upsilon \mathsf{p'q'r'}\vdash \mathsf{qpr}\Upsilon \mathsf{p'q'r'}\wedge \mathsf{pqr}\Upsilon \mathsf{q'p'r'}$ (permutation)
\item[(ii)] $\mathsf{p}\mathsf{q}\Upsilon \mathsf{r} \vdash \mathsf{p}\Upsilon \mathsf{rs}$ (monotonicity)
\item[(iii)] $\mathsf{p}\mathsf{q}\Upsilon \mathsf{r}\mathsf{q}\vdash \mathsf{p}\mathsf{q}\Upsilon \mathsf{r}$ (weakening)
\item[(iv)] $\mathsf{p}\Upsilon \langle\rangle\vdash\bot$
\end{description}
\end{example}
\begin{proof}
Items (i) and (iv) are clear. For item (ii), noting that $\mathsf{p}\Upsilon \mathsf{rs}:=\mathsf{p}\Upsilon \mathsf{r}\vee \mathsf{p}\Upsilon \mathsf{s}$, by $\vee\textsf{I}$ it suffices to show $\mathsf{p}\mathsf{q}\Upsilon \mathsf{r} \vdash\mathsf{p}\Upsilon \mathsf{r}$. Let $\mathsf{r}=\langle r_1\dots r_n\rangle$. By $\vee\textsf{E}$ and $\vee\textsf{I}$, it further suffices to show that for each $1\leq i\leq n$, $\mathsf{p}\mathsf{q}\Upsilon r_i\vdash \mathsf{p}\Upsilon r_i$, which is
\[\bigvee_{v\in 2^{\mathsf{K}},u\in 2^{\mathsf{M}}}\!(p_1^{v(1)} \wedge\dots\wedge p_k^{v(k)}\wedge q_1^{u(1)} \wedge\dots\wedge q_m^{u(m)}\wedge  \Upsilon r_i)\vdash\! \bigvee_{v\in 2^{\mathsf{K}}}(p_1^{v(1)} \wedge\dots\wedge p_k^{v(k)}\wedge   \Upsilon r_i),\]
where $\mathsf{K}=\{p_1,\dots,p_k\}$ and $\mathsf{M}=\{q_1,\dots,q_m\}$.
But this follows easily from $\wedge \textsf{E}$.

For item (iii), we show $\mathsf{p}\mathsf{q}\Upsilon \mathsf{r}\vee\mathsf{p}\mathsf{q}\Upsilon\mathsf{q}\vdash \mathsf{p}\mathsf{q}\Upsilon \mathsf{r}$. By $\vee \textsf{E}$ and \Cref{der_rule_prop}(\ref{der_rule_exfalso}), it suffices to show that $\mathsf{p}\mathsf{q}\Upsilon\mathsf{q}\vdash\bot$. Since $\mathsf{pq}\Upsilon\mathsf{q}=\mathsf{pq}\Upsilon q_1\vee\dots\vee \mathsf{pq}\Upsilon q_m$, it suffices to show that $\mathsf{pq}\Upsilon q_i\vdash\bot$ for each $1\leq i\leq m$, i.e.,
\begin{equation}\label{exm_anm_eq}
\bigvee_{v\in 2^{\mathsf{K}},u\in 2^{\mathsf{M}}}(p_1^{v(1)} \wedge\dots\wedge p_k^{v(k)}\wedge q_1^{u(1)} \wedge\dots\wedge q_m^{u(m)}\wedge (q_i\ror\neg q_i))\vdash \bot.
\end{equation}
Now, by \Cref{der_rule_prop}(\ref{der_rule_bote}) we have that $q_i^{u(i)}\wedge (q_i\ror\neg q_i)\vdash\bot$ for each $u\in 2^{\mathsf{M}}$. Thus, in  (\ref{exm_anm_eq}) each disjunct of the  formula on the left-hand-side of the turnstile implies $\bot$, from which we conclude that (\ref{exm_anm_eq}) holds by \sore.
%
\end{proof}

Finally, we give the proof of the normal form lemma, \Cref{DNF_PU}, which requires a few further lemmas.
The first one shows that a generalized version of the rule $\sor\textsf{E}_{\ror}$ with disjunctions of multiple formulas is derivable in our system.


\begin{lemma}\label{der_rule_sormulti}
For any nonempty index set $I$, we have that 
\[\Gamma,\bigsor_{i\in I}\phi_i\vdash\chi \iff \Gamma,\bigveedot_{i\in I_0}\phi_i\vdash\chi\text{ for all nonempty sets }I_0\subseteq I.\]
\end{lemma}
\begin{proof}
The direction from left to right follows easily from $\ror\sor$\textsf{Tr} and \sori. We prove the other direction by induction on $|I|$. The case $|I|= 1$ is  trivial. 
Now, if $I=J\cup\{k\}$, then we have that\allowdisplaybreaks
{\small\begin{align*}
&\quad\forall I_0\subseteq I,~I_0\neq\emptyset:\Gamma,\bigror_{i\in I_0}\phi_i\vdash\chi\\
\Longrightarrow&\quad\forall J_0\subseteq J,~J_0\neq\emptyset:\Gamma,\bigror_{i\in J_0}\phi_i\vdash\chi~~\&~~\Gamma,\phi_{k}\vdash\chi~\&~ \forall J_1\subseteq J,~J_1\neq\emptyset:\Gamma,\phi_k\ror\bigror_{j\in J_1}\phi_j\vdash\chi\\
\Longrightarrow&\quad\forall J_0\subseteq J,~J_0\neq\emptyset:\Gamma,\bigror_{i\in J_0}\phi_i\vdash\chi~~\&~~\Gamma,\phi_{k}\vdash\chi
~~\&~~ \forall J_1\subseteq J,~J_1\neq\emptyset:\Gamma,\bigror_{j\in J_1}(\phi_{k}\ror\phi_j)\vdash\chi\tag{since $\displaystyle\bigror_{j\in J_1}(\phi_{k}\ror\phi_j)\vdash\phi_k\ror\bigror_{j\in J_1}\phi_j$ by $\ror$\textsf{Ass}, $\ror$\textsf{Com} and \Cref{der_rule_prop}(\ref{der_rule_dstr_roridempotent})}\\
\Longrightarrow&\quad\Gamma,\bigvee_{i\in J}\phi_i\vdash\chi~~\&~~\Gamma,\phi_{k}\vdash\chi~~\&~~ \Gamma,\bigsor_{j\in J}(\phi_j\ror\phi_{k})\vdash\chi\tag{induction hypothesis}\\
\Longrightarrow&\quad\Gamma,\bigsor_{i\in J}\phi_i\vdash\chi~~\&~~\Gamma,\phi_{k}\vdash\chi~~~\&~~ \Gamma,(\bigsor_{j\in J}\phi_j)\ror\phi_{k}\vdash\chi\tag{$\dstr\ror\sor$}\\
\Longrightarrow&\quad\Gamma,(\bigsor_{i\in J}\phi_i)\sor\phi_{k}\vdash\chi\tag{$\sor\textsf{E}_{\ror}$}\\
\Longrightarrow&\quad\Gamma,\bigsor_{i\in I}\phi_i\vdash\chi\tag{since $I=J\cup\{k\}$}.
\end{align*}}
\end{proof}

Recall  that the formula $\Psi_X$ in the normal form defines the team $X$ modulo the empty team in the sense of Equation (\ref{eq_psi_x_prop}) from the proof of \Cref{PU_excmp} in Section \ref{sec:expr_comp}. Therefore for distinct teams $X$ and $Y$, the two formulas $\Psi_X$ and $\Psi_Y$ are contradictory to each other. We now prove this fact in our system, and  the proof of \Cref{DNF_PU} follows.

\begin{lemma}\label{der_rule_psiXYe}
If $X$ and $Y$ are two distinct  $\mathsf{N}$-teams, then $\Psi_X,\Psi_Y\vdash\phi$.
\end{lemma}
\begin{proof}
Let $\mathsf{N}=\{p_1,\dots,p_n\}$. If $X=\emptyset$ or $Y=\emptyset$, then $\Psi_X=\bot$ or $\Psi_Y=\bot$, and $\bot\vdash\phi$ follows from \Cref{der_rule_prop}(\ref{der_rule_exfalso}). Now assume that $X,Y\neq\emptyset$. Since $X\neq Y$, there exists (w.l.o.g.) some $v\in X\setminus Y$. By \Cref{der_rule_prop}(\ref{der_rule_bote}), we have $\Psi_{\{v\}}\ror\Psi_{X\setminus\{v\}}, \neg\Psi_{\{v\}}\vdash\phi$, i.e., $\Psi_X,\neg\Psi_{\{v\}}\vdash\phi$. To derive $\Psi_X,\Psi_Y\vdash\phi$ it then suffices to derive  $\Psi_Y\vdash\neg\Psi_{\{v\}}$. 
 By \Cref{der_rule_prop}(\ref{der_rule_dstr_rore}), this reduces to showing that for each $u\in Y$, $\Psi_{\{u\}}\vdash \neg\Psi_{\{v\}}$, which is equivalent (by the usual rules for classical formulas) to 
\[p_1^{u(1)},\dots, p_n^{u(n)}\vdash \neg p_1^{v(1)}\sor\dots\sor \neg p_n^{v(n)}.\]
 We have $u\neq v$ by the assumption, thus $p_i^{u(i)}=\neg p_i^{v(i)}$ for some $1\leq i\leq n$, from which and \sori the above clause follows.
\end{proof}

 

\begin{proof}[Proof of \Cref{DNF_PU}]
We  prove the lemma by induction on the complexity of $\phi$. 
If $\phi(p_{1},\dots,p_{n})=p_{i}$, then we can prove by the usual rules of classical formulas (which are all present or derivable in our system) that
\begin{align*}
p_{i}&\dashv\vdash\mathop{\bigsor_{v\in 2^{\mathsf{N}\setminus\{p_i\}}}}(p_{1}^{v(1)}\wedge\dots\wedge p_{i-1}^{v(i-1)}\wedge p_{i}\wedge p_{i+1}^{v(i+1)}\wedge\dots\wedge p_{n}^{v(n)})\\
&\dashv\vdash\bigvee_{\{u\}\in \mathcal{X}_i}\Psi_{\{u\}},\text{ where }\mathcal{X}_i=\{\{u\}\mid u\in  2^{\mathsf{N}},~u(i)=1\}.
\end{align*}
If $\phi=\bot$, then trivially $\bot\dashv\vdash\bigsor\emptyset=\bot$.  If $\phi=\top$,  we derive similarly by the  rules of classical formulas that
\[\top\dashv\vdash\bigsor_{v\in 2^{\mathsf{N}}}(p_{1}^{v(1)}\wedge\dots\wedge p_{n}^{v(n)})\dashv\vdash\bigvee_{\{v\}\in \mathcal{X}_\top}\Psi_{\{v\}},\text{ where }\mathcal{X}_\top=\{\{v\}\mid v\in 2^{\mathsf{N}}\}.\]


Suppose $\alpha(\mathsf{N})$ is a classical formula, and $\alpha\dashv\vdash\bigsor_{X\in\mathcal{X}}\Psi_{X}$.  We show that $\neg\alpha\dashv\vdash\bigsor_{v\in 2^{\mathsf{N}}\setminus\bigcup\mathcal{X}}\Psi_{\{v\}}$. 
It is sufficient to prove that $\bigsor_{X\in\mathcal{X}}\Psi_{X}\dashv\vdash\bigsor_{v\in \bigcup\mathcal{X}}\Psi_{\{v\}}$, which then implies, by the rules of negation $\neg$ and other usual rules  of classical formulas, that $\neg\alpha\dashv\vdash\neg\bigsor_{v\in \bigcup\mathcal{X}}\Psi_{\{v\}}\dashv\vdash\bigsor_{v\in 2^{\mathsf{N}}\setminus\bigcup\mathcal{X}}\Psi_{\{v\}}$. Now, 
we first have by the soundness theorem that $\alpha\dashv\vdash\bigsor_{X\in\mathcal{X}}\Psi_{X}$ implies that $\alpha\equiv\bigsor_{X\in\mathcal{X}}\Psi_{X}$. Then, observe that for each  $v\in \bigcup\mathcal{X}$, $\{v\}\in \mathcal{X}$. Indeed, by Equation (\ref{eq_psi_x_prop}) in Section \ref{sec:expr_comp}, it is easy to see that $\bigcup \mathcal{X}\models\bigsor_{X\in\mathcal{X}}\Psi_{X}$. Since the classical formula $\alpha$ is flat, we further have that $\{v\}\models \bigsor_{X\in\mathcal{X}}\Psi_{X}$, which by Equation (\ref{eq_psi_x_prop}) again implies that $\{v\}=X_0$ for some $X_0\in\mathcal{X}$, namely $\{v\}\in\mathcal{X}$.

 Thus, we derive $\Psi_{\{v\}}\vdash\bigsor_{X\in\mathcal{X}}\Psi_{X}$ by \sori. Hence we obtain $\bigsor_{v\in \bigcup\mathcal{X}}\Psi_{\{v\}}\vdash\bigsor_{X\in\mathcal{X}}\Psi_{X}$ by \sore. To prove the other direction,  for each $X\in\mathcal{X}$, since $X\subseteq\bigcup\mathcal{X}$, we derive by applying $\ror\sor$\textsf{Tr} and $\sor$\textsf{I} that
\[\Psi_X=\bigror_{u\in X}\Psi_{\{u\}}\vdash \bigsor_{u\in X}\Psi_{\{u\}}\vdash\bigsor_{v\in \bigcup\mathcal{X}}\Psi_{\{v\}}.
\]
Thus, we conclude that $\bigsor_{X\in\mathcal{X}}\Psi_{X}\vdash\bigsor_{v\in \bigcup\mathcal{X}}\Psi_{\{v\}}$ by applying $\sor$\textsf{E}.

\vspace{0.5\baselineskip}

Suppose $\psi(\mathsf{N})$ and $\chi(\mathsf{N})$ satisfy $\psi\dashv\vdash\bigsor_{X\in\mathcal{X}}\Psi_{X}$ and $\chi\dashv\vdash\bigsor_{Y\in\mathcal{Y}}\Psi_{Y}$,
for some  finite sets $\mathcal{X}$ and $\mathcal{Y}$ of $\mathsf{N}$-teams. The case $\phi=\psi\vee\chi$ is clear.
If $\phi=\psi\ror\chi$, and $\mathcal{X}=\emptyset$ or $\mathcal{Y}=\emptyset$, i.e., $\psi\dashv\vdash \bot$ or $\chi\dashv\vdash \bot$, then we derive $\psi\ror\chi\dashv\vdash\bot=\bigsor\emptyset$ by $\ror\textsf{Sub}$, $\ror\bot$\textsf{E} and \Cref{der_rule_prop}(\ref{der_rule_exfalso}). 
If $\mathcal{X},\mathcal{Y}\neq\emptyset$, we show that $\psi\ror\chi\dashv\vdash \bigsor_{X\in\mathcal{X},Y\in\mathcal{Y}}\Psi_{X\cup Y}$.
For the left to right direction, we have that 
\begin{align*}
\psi\ror\chi&\vdash\Big(\bigsor_{X\in\mathcal{X}}\Psi_{X}\Big)\ror\Big(\bigsor_{Y\in\mathcal{Y}}\Psi_{Y}\Big)\tag{by induction hypothesis and $\ror\textsf{Mon}$}\\
&\vdash\bigsor_{X\in\mathcal{X}}\Big(\Psi_{X}\ror\Big(\bigsor_{Y\in\mathcal{Y}}\Psi_{Y}\Big)\Big)\tag{\dstr$\ror\sor$}\\
&\vdash\bigsor_{X\in\mathcal{X}}\bigsor_{Y\in\mathcal{Y}}\left(\Psi_{X}\ror\Psi_{Y}\right)\tag{\dstr$\ror\sor$}\\
&\vdash\displaystyle\bigsor_{X\in\mathcal{X},Y\in\mathcal{Y}}\Psi_{X\cup Y}\tag{apply Prop. \ref{der_rule_prop}(\ref{der_rule_dstr_roridempotent}) and \sore for the case $X=Y$}.
\end{align*}
The other direction  is proved similarly using \rori and \Cref{der_rule_prop}(\ref{der_rule_dstr_rorsor}).

\vspace{0.5\baselineskip}

If $\phi=\psi\wedge\chi$, and $\mathcal{X}=\emptyset$ or $\mathcal{Y}=\emptyset$, i.e., $\psi\dashv\vdash \bot$ or $\chi\dashv\vdash \bot$, then we derive $\psi\wedge\chi\dashv\vdash\bot=\bigsor\emptyset$ by \ce and \Cref{der_rule_prop}(\ref{der_rule_exfalso}).  If $\mathcal{X},\mathcal{Y}\neq\emptyset$, we show that $\psi\wedge\chi\dashv\vdash\bigsor_{Z\in \mathcal{Z}}\Psi_Z$, where
\[\mathcal{Z}=\{\bigcup \mathcal{X}'\mid  \mathcal{X}'\subseteq \mathcal{X}\text{ and }\bigcup \mathcal{X}'=\bigcup \mathcal{Y}'\text{ for some }\mathcal{Y}'\subseteq \mathcal{Y}\}.\]
For the right to left direction, by \sore it suffices to derive $\Psi_Z\vdash\psi\wedge\chi$ for each $Z=\bigcup \mathcal{X}'=\bigcup \mathcal{Y}'\in \mathcal{Z}$, where $\mathcal{X}'\subseteq \mathcal{X}$ and $\mathcal{Y}'\subseteq \mathcal{Y}$.  By \Cref{lem_rorsor_trans}, we have that $\Psi_Z\vdash\bigvee_{X\in \mathcal{X}'}\Psi_X$. Further, by \sori and the induction hypothesis we derive $\bigvee_{X\in \mathcal{X}'}\Psi_X\vdash\bigvee_{X\in \mathcal{X}}\Psi_X\vdash\psi$. Hence, $\Psi_Z\vdash\psi$. The fact $\Psi_Z\vdash\chi$ is proved similarly.

For the left to right direction, by induction hypothesis and \Cref{der_rule_sormulti} it suffices to prove that for each nonempty $\mathcal{X}'\subseteq \mathcal{X}$ and  $\mathcal{Y}'\subseteq \mathcal{Y}$,
\[\bigror_{X\in \mathcal{X}'}\Psi_X,\bigror_{Y\in \mathcal{Y}'}\Psi_Y\vdash\bigsor_{Z\in\mathcal{Z}}\Psi_Z.\]
Note that elements in $\mathcal{X}'$ and in $\mathcal{Y}'$ may not be disjoint. So by \Cref{der_rule_prop}(\ref{der_rule_dstr_roridempotent}) we further reduce showing the above clause to showing $\Psi_{\bigcup \mathcal{X}'},\Psi_{\bigcup \mathcal{Y}'}\vdash\bigsor_{Z\in\mathcal{Z}}\Psi_Z$.
But now, if $\bigcup \mathcal{X}'\neq\bigcup \mathcal{Y}'$,  the desired clause follows simply from \Cref{der_rule_psiXYe}. Otherwise, if $\bigcup \mathcal{X}'=\bigcup \mathcal{Y}'\in \mathcal{Z}$, then we have $\Psi_{\bigcup \mathcal{X}'}\vdash\bigsor_{Z\in\mathcal{Z}}\Psi_Z$ by \sori. 
 \end{proof}

\subsection{\PIncz}

In this subsection, we axiomatize the sublogic \PIncz of \PInc, by introducing a sound and complete system of natural deduction. Recall that  \PIncz contains inclusion atoms of  primitive form $\mathsf{x}\subseteq\mathsf{a}$ with $x_i\in \{\top,\bot\}$ only. The system of \PInc will be introduced in the next subsection as an extension of the one for \PIncz. The proof of the completeness theorem for the system of \PIncz applies essentially the same argument (via normal form) as that in the previous subsection for \PU. 
Since the normal form of \PIncz (given in \Cref{nf-pinc-picnst}(\ref{nf-pinc-picnst-1})) is more complex, the proofs in this subsection will involve more steps. Let us start, again, by presenting the deduction system.

%


\begin{table}[t]
\begin{center}
\caption{Rules for (primitive) inclusion atoms}
\vspace{-4pt}
\setlength{\tabcolsep}{2pt}
\renewcommand{\arraystretch}{1.2}
\setlength{\extrarowheight}{0.5pt}
\scalebox{.92}{\begin{tabular}{|C{0.33\linewidth}C{0.33\linewidth}C{0.33\linewidth}|}
\hline
~\AxiomC{}\noLine\UnaryInfC{{\small$D$}}\noLine\UnaryInfC{$\mathsf{a}\mathsf{b}\mathsf{c}\subseteq \mathsf{a'}\mathsf{b'}\mathsf{c'}$}\RightLabel{\incexc}\UnaryInfC{$\mathsf{b}\mathsf{a}\mathsf{c}\subseteq \mathsf{b'}\mathsf{a'}\mathsf{c'}$}\noLine\UnaryInfC{}\DisplayProof
&\AxiomC{}\noLine\UnaryInfC{{\small$D$}}\noLine\UnaryInfC{$\mathsf{a}\mathsf{b}\subseteq \mathsf{c}\mathsf{d}$}\RightLabel{\incctr}\UnaryInfC{$\mathsf{a}\subseteq \mathsf{c}$}\noLine\UnaryInfC{}\DisplayProof
&\AxiomC{}\noLine\UnaryInfC{{\small$D$}}\noLine\UnaryInfC{$\mathsf{ab}\subseteq\mathsf{cd}$}\RightLabel{$\subseteq\textsf{Wk}$}\UnaryInfC{$\mathsf{aab}\subseteq\mathsf{ccd}$}\noLine\UnaryInfC{}\DisplayProof
\\
~\AxiomC{{\small$D_0$}}\noLine\UnaryInfC{$\mathsf{a}\subseteq\mathsf{b}$}\AxiomC{{\small$D_1$}}\noLine\UnaryInfC{$\mathsf{b}\subseteq\mathsf{c}$}\RightLabel{\inctrs}\BinaryInfC{$\mathsf{a}\subseteq\mathsf{c}$}\noLine\UnaryInfC{} \DisplayProof
&\AxiomC{}\noLine\UnaryInfC{}\RightLabel{\incid} \UnaryInfC{$\mathsf{a}\subseteq\mathsf{a}$}\noLine\UnaryInfC{}   \DisplayProof 
&\AxiomC{{\small$D_0$}}\noLine\UnaryInfC{$\mathsf{a}\subseteq\mathsf{b}$}\AxiomC{{\small$D_1$}}\noLine\UnaryInfC{$\alpha(\mathsf{b})$}\RightLabel{\inccmp}\BinaryInfC{$\alpha(\mathsf{a})$}\noLine\UnaryInfC{} \DisplayProof
\\
\multicolumn{3}{|c|}{\!\!\AxiomC{{\small$D$}}\noLine\UnaryInfC{$\mathsf{a}\subseteq\mathsf{b}$}\RightLabel{\inczext\!\!} \UnaryInfC{$\mathsf{\top a}\subseteq\mathsf{\top b}$}\noLine\UnaryInfC{}   \DisplayProof
\!\!\!\!\AxiomC{{\small$D$}}\noLine\UnaryInfC{$\mathsf{a}\subseteq\mathsf{b}$}\RightLabel{\inczext\!\!} \UnaryInfC{$\mathsf{\bot a}\subseteq\mathsf{\bot b}$}\noLine\UnaryInfC{}   \DisplayProof
\!\!\!\!\!\!\AxiomC{{\small$D_0$}}\noLine\UnaryInfC{$p$} \AxiomC{{\small$D_1$}}\noLine\UnaryInfC{$\mathsf{a}\subseteq\mathsf{b}$}  \RightLabel{\inczext\!\!}\BinaryInfC{$\top\mathsf{a}\subseteq p\mathsf{b}$}\noLine\UnaryInfC{} \DisplayProof
\!\!\!\!\AxiomC{{\small$D_0$}}\noLine\UnaryInfC{$\neg p$} \AxiomC{{\small$D_1$}}\noLine\UnaryInfC{$\mathsf{a}\subseteq\mathsf{b}$}  \RightLabel{\inczext\!\!}\BinaryInfC{$\bot\mathsf{a}\subseteq p\mathsf{b}$}\noLine\UnaryInfC{} \DisplayProof}\\
\multicolumn{3}{|c|}{\!\!\!\!\!\!\!\!\!\!\!\!\AxiomC{{\small$D$}}\noLine\UnaryInfC{$(\phi\wedge \mathsf{x}\subseteq\mathsf{a})\vee\psi$} \AxiomC{}\noLine\UnaryInfC{~~~[$\phi$]~~~~~~~~~~~~~[$\mathsf{x}\subseteq\mathsf{a}$]}\noLine
\UnaryInfC{{\small$D_0$}}
\branchDeduce
\DeduceC{$\chi$} 
\AxiomC{}\noLine\UnaryInfC{\raisebox{4.5pt}{[$\psi$]}}\noLine\UnaryInfC{}\noLine\UnaryInfC{{\small$D_1$}}\noLine\UnaryInfC{}\noLine\UnaryInfC{}\noLine\UnaryInfC{}\noLine\UnaryInfC{$\chi$} 
\AxiomC{}\noLine\UnaryInfC{[$\phi\vee\psi$]~~~~~~~ [$\mathsf{x}\subseteq\mathsf{a}$]} 
\noLine
\UnaryInfC{{\small$D_2$}}
\branchDeduce
\DeduceC{$\chi$} 
\RightLabel{$\vee_{\subseteq_0}\textsf{E}$}\QuaternaryInfC{$\chi$}\noLine\UnaryInfC{}\DisplayProof~~~~~~~}\\
\multicolumn{3}{|c|}{\AxiomC{{\small$D_0$}}\noLine\UnaryInfC{$\phi\vee\psi$}
\AxiomC{{\small$D_1$} ~~ $\dots$ ~~{\small$D_k$}}\noLine\UnaryInfC{$\mathsf{x}_1\subseteq\mathsf{a}_1$~ $\dots$ ~$\mathsf{x}_k\subseteq\mathsf{a}_k$}
\RightLabel{$\subseteq_0\!\textsf{Dst}$}\BinaryInfC{$\big((\phi\vee \mathsf{a}_1^{\mathsf{x}_1}\vee\dots\vee \mathsf{a}_k^{\mathsf{x}_k})\wedge \mathsf{x}_1\subseteq\mathsf{a}_1\wedge\dots\wedge \mathsf{x}_k\subseteq\mathsf{a}_k\big)\vee\psi$}\noLine\UnaryInfC{} \DisplayProof }\\
\hline
\end{tabular}}
\label{rules_PIncz}
\vspace{-10pt}
\end{center}
\end{table}

\begin{definition}
The system of \PIncz consists of the rules for constants and connectives in \Cref{tab_rule_general} and the rules for inclusion atoms in \Cref{rules_PIncz}, where  $\alpha$ ranges over classical formulas only, $\mathsf{a},\mathsf{b},\mathsf{c},\dots$ (with or without subscripts) are arbitrary (and possibly empty) sequences of elements in $\textsf{Prop}\cup\{\top,\bot\}$,  $\mathsf{x}$ (with or without subscripts) stands for an arbitrary sequence of constants $\top$ and $\bot$, and the notation $\alpha(\mathsf{a})$ indicates that the propositional variables and constants occurring in $\alpha$ are among $\mathsf{a}$. 
\end{definition}


All the rules except for the last two in \Cref{rules_PIncz} are actually sound also for arbitrary inclusion atoms (that are not necessarily primitive).
It was proved in \cite{inclusion_dep_CFP_82} that the implication problem of inclusion dependencies 
is completely axiomatized by the rules \incid and \inctrs together with the following projection rule:
\[\AxiomC{}\noLine\UnaryInfC{$a_1\dots a_k\subseteq b_1\dots b_k$}  \RightLabel{$\subseteq$\textsf{Proj}}\UnaryInfC{$a_{i_1}\dots a_{i_m}\subseteq b_{i_1}\dots b_{i_m}$}\noLine\UnaryInfC{} \DisplayProof  ~~(i_1,\dots, i_m\in\{1,\dots,k\}).\]
This rule $\subseteq$\textsf{Proj} is easily shown to be equivalent to  the three rules \incexc, \incctr, and $\subseteq\textsf{Wk}$  together in our system.


The inclusion atom compression rule \inccmp is a natural generalization of a similar rule introduced in \cite{Hannula_fo_ind_13} for first-order inclusion atoms. The primitive inclusion atom extension rule \inczext in four different forms are evidently sound. Note however that a stronger form of the extension rule $\mathsf{a}\subseteq\mathsf{b}/p\mathsf{a}\subseteq p\mathsf{b}$ with $p$ being a propositional variable in the context of the full logic \PInc is easily seen to be {\em not} sound.\todo{rephrase}

The rule $\vee_{\subseteq_0}\textsf{E}$ simulates  the entailment 
\begin{equation}\label{vincE_sound}
(\phi\wedge\mathsf{x}\subseteq\mathsf{a})\vee\psi\models (\phi\wedge\mathsf{x}\subseteq\mathsf{a})\vvee \psi\vvee ((\phi\vee\psi)\wedge \mathsf{x}\subseteq\mathsf{a}),
\end{equation}
which highlights the fact that in a team $X$ satisfying the formula $(\phi\wedge\mathsf{x}\subseteq\mathsf{a})\vee\psi$, if the  left disjunct of the formula  is satisfied by a nonempty subteam of $X$, then the primitive inclusion atom $\mathsf{x}\subseteq\mathsf{a}$ (being upward closed) is true actually in the whole team $X$.
Note that the converse direction of the entailment (\ref{vincE_sound}) does not hold, because the third disjunct $(\phi\vee\psi)\wedge \mathsf{x}\subseteq\mathsf{a}$ of the formula on the right-hand-side does not necessarily imply $(\phi\wedge\mathsf{x}\subseteq\mathsf{a})\vee\psi$. Instead, the  formula $(\phi\vee\psi)\wedge \mathsf{x}\subseteq\mathsf{a}$  implies $((\phi\vee\mathsf{a}^{\mathsf{x}})\wedge\mathsf{x}\subseteq\mathsf{a})\vee\psi$, as the rule $\subseteq_0\!\textsf{Dst}$ states.



\begin{theorem}[Soundness]
For any set $\Gamma\cup\{\phi\}$ of \PIncz-formulas,
we have that $\Gamma\vdash\phi\Longrightarrow\Gamma\models\phi$.
\end{theorem}
\begin{proof}
The soundness of the rule $\vee_{\subseteq_0}\textsf{E}$ follows from the entailment (\ref{vincE_sound}), which can be easily checked. We now verify the soundness of the rules $\subseteq\!\textsf{Cmp}$ and $\subseteq_0$\textsf{Dst}.  The other rules are easily seen to be sound (for primitive and also arbitrary inclusion atoms).

For $\subseteq\!\textsf{Cmp}$, we will verify its soundness for arbitrary inclusion atoms (that are not necessarily primitive). Suppose $X\models\mathsf{a}\subseteq\mathsf{b}$ and $X\models\alpha(\mathsf{b})$. By \Cref{cpl_singleton_single}, to show that $X\models\alpha(\mathsf{a})$ it suffices to show $v\models\alpha(\mathsf{a})$ for any $v\in X$. By the assumption, there exists $u\in X$ such that $u(\mathsf{b})=v(\mathsf{a})$. Since $X\models\alpha(\mathsf{b})$ and  $\alpha(\mathsf{b})$ is flat, we have that $u\models\alpha(\mathsf{b})$. Now, if all elements in the sequences $\mathsf{a},\mathsf{b}$ are propositional variables, then $v\models\alpha(\mathsf{a})$  follows from locality. In case some elements in $\mathsf{a},\mathsf{b}$ are constants $\top$ or $\bot$, replace these elements with fresh propositional variables to obtain two new sequences $\mathsf{a}',\mathsf{b}'$ of propositional variables. Let $v',u'$ be valuations for $\mathsf{a}',\mathsf{b}'$ that agree with $v,u$ respectively on all propositional variables from $\mathsf{a},\mathsf{b}$, and map the fresh propositional variable corresponding to $\top$ to $1$ and the fresh propositional variable corresponding to $\bot$ to $0$. Clearly, $u'(\mathsf{b}')=u(\mathsf{b})=v(\mathsf{a})=v'(\mathsf{a}')$. Thus, by properties of classical propositional logic, we have that
\[v\models\alpha(\mathsf{a})\text{ iff }v'\models\alpha(\mathsf{a}'),\text{ and  }u\models\alpha(\mathsf{b})\text{ iff }u'\models\alpha(\mathsf{b}').\]
Now,  the fact that  $u\models\alpha(\mathsf{b})$ then implies  $u'\models\alpha(\mathsf{b}')$, which further implies, by locality, that $v'\models\alpha(\mathsf{a}')$. Thus, $v\models\alpha(\mathsf{a})$ follows.

For $\subseteq_0$\textsf{Dst}, suppose that $X\models\phi\vee\psi$ and $X\models\mathsf{x}_1\subseteq\mathsf{a}_1\wedge\dots\wedge\mathsf{x}_k\subseteq\mathsf{a}_k$. The latter implies that there are $v_1,\dots,v_k\in X$ such that $v_1(\mathsf{a}_1)=v_1(\mathsf{x}_1),\dots,v_k(\mathsf{a}_k)=v_k(\mathsf{x}_k)$. Thus, $\{v_1\}\models\mathsf{a}_1^{\mathsf{x}_1},\dots,\{v_k\}\models\mathsf{a}_k^{\mathsf{x}_k}$. On the other hand, there are $Y,Z\subseteq X$ such that $X=Y\cup Z$, $Y\models\phi$ and $Z\models\psi$. Clearly, $Y\cup\{v_1,\dots,v_k\}\models \phi\vee\mathsf{a}_1^{\mathsf{x}_1}\vee\dots\vee\mathsf{a}_k^{\mathsf{x}_k}$ and $Y\cup\{v_1,\dots,v_k\}\models\mathsf{x}_1\subseteq\mathsf{a}_1\wedge\dots\wedge\mathsf{x}_k\subseteq\mathsf{a}_k$. Hence, we conclude that $Y\cup\{v_1,\dots,v_k\}\cup Z=X\models\big((\phi\vee\mathsf{a}_1^{\mathsf{x}_1}\vee\dots\vee\mathsf{a}_k^{\mathsf{x}_k})\wedge \mathsf{x}_1\subseteq\mathsf{a}_1\wedge\dots\wedge\mathsf{x}_k\subseteq\mathsf{a}_k\big)\vee\psi$.
\end{proof}

The proof of the completeness theorem uses a similar normal form argument to that in the previous subsection. The general structure and key ingredients of the proof are the same as in the completeness proof for the system of \PU. The crucial step is to show the following lemma that every formula in \PIncz is provably equivalent to a formula in the normal  form $\bigvee_{X\in\mathcal{X}}(\Theta_X\wedge \Phi_X)$ (see also \Cref{nf-pinc-picnst}(\ref{nf-pinc-picnst-1})). Since the normal form for \PIncz is more complex, the proof of this lemma involves more preparation steps. As in the previous section, we will only give the detailed proof of the lemma after we presented the completeness proof (in \Cref{PInc_completeness}).

\begin{lemma}\label{DNF_PInc}
Let $\mathsf{N}=\{p_1,\dots,p_n\}$. Every \PIncz-formula $\phi(\mathsf{N})$ is provably equivalent to a formula of the form $\displaystyle \bigsor_{X\in\mathcal{X}}(\Theta_{X}\wedge\Phi_X)$, where  $\mathcal{X}$ is a  finite set of $\mathsf{N}$-teams,
\begin{equation}\label{NF_PInc_prof_eq}
\Theta_X:=\bigsor_{v\in X}(p_{1}^{v(1)}\wedge\dots\wedge p_{n}^{v(n)}),\quad\text{and}\quad\Phi_X:=\bigwedge_{v\in X}\underline{v(1)}\dots \underline{v(n)}\subseteq p_{1}\dots p_{n},
\end{equation}
\end{lemma}

Recall from the proof of \Cref{PT_exp_pw}  that each disjunct $\Theta_X\wedge \Phi_X$ of the normal form $\bigsor_{X\in\mathcal{X}}(\Theta_{X}\wedge\Phi_X)$ above is equivalent to the \PU-formula $\Psi_X$. Thus, \Cref{comp_main_lm} from the previous subsection with respect to the formulas $\Psi_X$ holds also with $\Theta_X\wedge\Phi_X$ in place of $\Psi_X$.

Another key lemma for the completeness theorem is the following fact that corresponds to a specific case of Lemma \ref{lem_rorsor_trans} in the previous subsection.

%
%
%

\begin{lemma}\label{PInc_dstr_lm}
For any finite set $\mathcal{Y}$ of $\mathsf{N}$-teams, $\Theta_{\bigcup\mathcal{Y}},\Phi_{\bigcup\mathcal{Y}}\vdash \bigvee_{Y\in \mathcal{Y}}(\Theta_Y\wedge\Phi_{Y})$.
\end{lemma}
\begin{proof}
Let $\mathcal{Y}=\{Y_1,\dots,Y_k\}$. We first derive that 
\begin{align*}
\Theta_{Y_1\cup\dots\cup Y_k},\Phi_{Y_1}&\vdash\big(\Theta_{Y_1}\vee(\Theta_{Y_2}\vee\dots\vee\Theta_{Y_k})\big)\wedge\Phi_{Y_1}\tag{\sori}\\
&\vdash \big((\Theta_{Y_1}\vee\Theta_{Y_1})\wedge \Phi_{Y_1}\big)\vee(\Theta_{Y_2}\vee\dots\vee\Theta_{Y_k})\tag{$\subseteq_0\!\textsf{Dst}$}\\
&\vdash  (\Theta_{Y_1}\wedge \Phi_{Y_1})\vee(\Theta_{Y_2}\vee\dots\vee\Theta_{Y_k})\tag{\sore}
\end{align*}
Similarly, we have that
\[(\Theta_{Y_1}\wedge \Phi_{Y_1})\vee(\Theta_{Y_2}\vee\dots\vee\Theta_{Y_k}),\Phi_{Y_2}\vdash (\Theta_{Y_2}\wedge \Phi_{Y_2})\vee((\Theta_{Y_1}\wedge\Phi_{Y_1})\vee\Theta_{Y_3}\vee\dots\vee\Theta_{Y_k})\]
and so on. In the end, putting all these steps together, we obtain that 
\[\Theta_{Y_1\cup\dots\cup Y_k},\Phi_{Y_1},\dots,\Phi_{Y_k}\vdash (\Theta_{Y_1}\wedge\Phi_{Y_1})\vee\dots\vee(\Theta_{Y_k}\wedge\Phi_{Y_k}).\]
\end{proof}

\begin{theorem}[Completeness]\label{PInc_completeness}
For any set $\Gamma\cup\{\phi\}$ of \PIncz-formulas, we have that $\Gamma\models\phi \iff \Gamma\vdash\phi$.  
\end{theorem}
\begin{proof}
The proof follows from  the same argument as  that for \Cref{PU_completeness} with the \PU-formula $\Psi_X$ now replaced with the \PIncz-formula $\Theta_X\wedge\Phi_X$. In the key steps, we now apply \Cref{DNF_PInc} and \Cref{comp_main_lm} for $\Theta_X\wedge\Phi_X$. Another crucial fact that $\Theta_X\wedge\Phi_X\vdash\bigvee_{Y\in \mathcal{Y}_X}(\Theta_Y\wedge\Phi_Y)$ in case $X=\bigcup\mathcal{Y}_X$ for some $\mathcal{Y}_X\subseteq \mathcal{Y}$ is  given by \Cref{PInc_dstr_lm}.
\end{proof}

The rest of this subsection is devoted to the proof of \Cref{DNF_PInc}.
We first prove the following technical lemma.  


\begin{proposition}\label{der_rule_prop_pinc}
\begin{enumerate}[(i)]
\item\label{der_rule_pinc_ince} $\neg (a_1^{x_1}\wedge\dots\wedge a_n^{x_n}), x_1\dots x_n\subseteq a_{1}\dots a_{n}\vdash\bot$.
\item\label{der_rule_pinc_winci} $p_1^{x_1},\dots,p_n^{x_n}\vdash x_1\dots x_n\subseteq p_{1}\dots p_{n}$. 
\item\label{der_rule_pinc_allposs} $\vdash \bigsor_{v\in 2^{\mathsf{N}}}(\Theta_v\wedge \Phi_v)$ for any $\mathsf{N}\subseteq\mathsf{Prop}$.
\end{enumerate}
\end{proposition}
\begin{proof}
For item (\ref{der_rule_pinc_ince}), we derive by $\subseteq\textsf{Cmp}$ that 
\[\neg (a_1^{x_1}\wedge\dots\wedge a_n^{x_n}), x_1\dots x_n\subseteq a_{1}\dots a_{n}\vdash\neg (x_1^{x_1}\wedge\dots\wedge x_n^{x_n})\vdash\neg (\top\wedge\dots\wedge\top)\vdash\bot.\] 

For item (\ref{der_rule_pinc_winci}), we derive by applying \inczext and \incctr that 
\[p_1^{x_1},\top\subseteq \top\vdash x_1\top\subseteq p_1 \top\vdash x_1\subseteq p_1.\]
Since $\vdash \top\subseteq \top$ by  \incid , we conclude $p_1^{x_1}\vdash  x_1 \subseteq p_1$. By \inczext again, we  derive $p_2^{x_2},x_1 \subseteq p_1\vdash x_2x_1\subseteq p_1p_2$, and thus $p_1^{x_1},p_2^{x_2}\vdash x_2x_1\subseteq p_1p_2$. Proceed in the same way we obtain $p_1^{x_1},\dots,p_n^{x_n}\vdash x_1\dots x_n\subseteq p_{1}\dots p_{n}$ in the end.
%
%

For item (\ref{der_rule_pinc_allposs}), we first derive by rules of classical formulas that $\vdash \bigsor_{v\in 2^{\mathsf{N}}}\Theta_v$.  For each $v\in 2^{\mathsf{N}}$, by  item (\ref{der_rule_pinc_winci}) we have that $\Theta_v\vdash\Phi_v\vdash\Theta_v\wedge\Phi_v$. Hence we conclude $\vdash \bigsor_{v\in 2^{\mathsf{N}}}(\Theta_v\wedge \Phi_v)$ by applying \sori and \sore.
\end{proof}

Next, we show that the rule $\vee_{\subseteq_0}\textsf{E}$ for single primitive inclusion atoms can be generalized to one with multiple primitive inclusion atoms, and further to one with multiple disjunctions.

\begin{lemma}\label{veesub_eli}
\begin{enumerate}[(i)]
\item Let $\mathsf{x}_1\subseteq\mathsf{a}_1,\dots, \mathsf{x}_k\subseteq\mathsf{a}_k$ be primitive inclusion atoms. If 
\[\Gamma,\phi,\mathsf{x}_1\subseteq\mathsf{a}_1,\dots, \mathsf{x}_k\subseteq\mathsf{a}_k\vdash\chi,~~ \Gamma,\psi\vdash\chi,~~\text{and }\Gamma,\phi\vee\psi,\mathsf{x}_1\subseteq\mathsf{a}_1,\dots, \mathsf{x}_k\subseteq\mathsf{a}_k\vdash\chi,\] 
then $\Gamma,(\phi\wedge\mathsf{x}_1\subseteq\mathsf{a}_1\wedge\dots\wedge \mathsf{x}_k\subseteq\mathsf{a}_k)\vee\psi\vdash\chi$.
\item\label{veesub_eli_c} Let $I$ be a nonempty finite index set. For each $i\in I$, let $\iota_i$ be the conjunction of some  finitely many primitive inclusion atoms. If for every nonempty $J\subseteq I$,
\begin{equation}\label{veesub_eli_eq2}
\Gamma,\bigvee_{i\in J}\phi_i,\bigwedge_{i\in J}\iota_i\vdash\chi,
\end{equation}
then $\Gamma,\bigvee_{i\in I}(\phi_i\wedge\iota_i)\vdash\chi$.
\end{enumerate}
\end{lemma}
\begin{proof}
(i). To show that $\Gamma,\big((\phi\wedge\mathsf{x}_2\subseteq\mathsf{a}_2\wedge\dots\wedge \mathsf{x}_k\subseteq\mathsf{a}_k)\wedge \mathsf{x}_1\subseteq\mathsf{a}_1\big)\vee\psi\vdash\chi$, by $\vee_{\subseteq_0}\textsf{E}$ it suffices to show that 
\[\Gamma,\phi,\mathsf{x}_2\subseteq\mathsf{a}_2,\dots, \mathsf{x}_k\subseteq\mathsf{a}_k,\mathsf{x}_1\subseteq\mathsf{a}_1\vdash\chi,\quad \Gamma,\psi\vdash\chi, \]
\[\text{and }\Gamma,(\phi\wedge\mathsf{x}_2\subseteq\mathsf{a}_2\wedge\dots\wedge \mathsf{x}_k\subseteq\mathsf{a}_k)\vee\psi,\mathsf{x}_1\subseteq\mathsf{a}_1\vdash\chi.\]
The first two clauses are given already by the assumption. To prove the third clause, by $\vee_{\subseteq_0}\textsf{E}$ again, it suffices to prove that
\[\Gamma,\phi,\mathsf{x}_3\subseteq\mathsf{a}_3,\dots, \mathsf{x}_k\subseteq\mathsf{a}_k,\mathsf{x}_1\subseteq\mathsf{a}_1,\mathsf{x}_2\subseteq\mathsf{a}_2\vdash\chi,\quad \Gamma,\psi,\mathsf{x}_1\subseteq\mathsf{a}_1\vdash\chi, \]
\[\text{and }\Gamma,(\phi\wedge\mathsf{x}_3\subseteq\mathsf{a}_3\wedge\dots\wedge \mathsf{x}_k\subseteq\mathsf{a}_k)\vee\psi,\mathsf{x}_1\subseteq\mathsf{a}_1,\mathsf{x}_2\subseteq\mathsf{a}_2\vdash\chi.\]
Again, the first two clauses follow from the assumption, and the third clause can be reduced to simpler clauses by applying $\vee_{\subseteq_0}\textsf{E}$. Proceed this way, in the end it remains to show that 
\(\Gamma,\phi\vee\psi,\mathsf{x}_1\subseteq\mathsf{a}_1,\dots,\mathsf{x}_k\subseteq\mathsf{a}_k\vdash\chi.\)
But this is also given by the assumption, and we are then done.

(ii). Suppose (\ref{veesub_eli_eq2}) holds for all nonempty $J\subseteq I$. We first prove a lemma that for any disjoint $K,L\subseteq I$ with $K\neq\emptyset$,
\begin{equation}\label{veesub_eli_eq3}
\Gamma,\bigvee_{k\in K}\phi_k\vee\bigvee_{l\in L}(\phi_l\wedge\iota_l),\bigwedge_{k\in K}\iota_k\vdash\chi.
\end{equation}
We proceed by induction on $|L|$. If $L=\emptyset$, then $\Gamma,\bigvee_{k\in K}\phi_k,\bigwedge_{k\in K}\iota_k\vdash\chi$ is given by assumption (since $K\neq\emptyset$). Suppose the claim holds for $L$. We show that 
\[\Gamma,\bigvee_{k\in K}\phi_k\vee(\bigvee_{l\in L}(\phi_l\wedge\iota_l))\vee(\phi_0\wedge\iota_0),\bigwedge_{k\in K}\iota_k\vdash\chi.\]
By item (i), it suffices to show that 
\[\Gamma,\phi_0,\iota_0,\bigwedge_{k\in K}\iota_k\vdash\chi,\quad \Gamma,\bigvee_{k\in K}\phi_k\vee\bigvee_{l\in L}(\phi_l\wedge\iota_l),\bigwedge_{k\in K}\iota_k\vdash\chi\]
\[\text{and }\Gamma,\phi_0\vee\bigvee_{k\in K}\phi_k\vee\bigvee_{l\in L}(\phi_l\wedge\iota_l),\bigwedge_{k\in K}\iota_k,\iota_0\vdash\chi.\]
The first clause follows from the assumption that $\Gamma,\phi_0,\iota_0\vdash\chi$. The last two clauses follow from the induction hypothesis.

Now we prove $\Gamma,\bigvee_{i\in I}(\phi_i\wedge\iota_i)\vdash\chi$ by induction on $|I|$. If $|I|=1$, the claim trivially holds.
Suppose the claim holds for $I$. We show that the claim holds also for $I\cup\{0\}$, that is $\Gamma,(\bigvee_{i\in I}(\phi_{i}\wedge\iota_{i}))\vee(\phi_{0}\wedge\iota_{0})\vdash\chi$ holds, assuming that (\ref{veesub_eli_eq2}) holds for any $J\subseteq I\cup\{0\}$. By item (i), it suffices to show that
\[\Gamma,\bigvee_{i\in I}(\phi_{i}\wedge\iota_{i})\vdash\chi,\quad \Gamma,\phi_{0},\iota_{0}\vdash\chi,\quad \Gamma,(\bigvee_{i\in I}(\phi_{i}\wedge\iota_{i}))\vee\phi_0,\iota_0\vdash\chi.\]
The second clause is given by the assumption. The first clause follows from the induction hypothesis, since  for every subset $J\subseteq I\subseteq I\cup\{0\}$, (\ref{veesub_eli_eq2}) holds by assumption. The third clause follows from (\ref{veesub_eli_eq3}).
%
\end{proof}

We call a primitive inclusion atom $\mathsf{x}\subseteq \mathsf{p}$ {\em regular} if the $\mathsf{p}$ is a sequence of distinct propositional variables. For example, the primitive inclusion atoms $\bot\top\bot\subseteq ppq$ and $\top\bot\bot\subseteq pq\bot$ are not regular. We now show that every nontrivial primitive inclusion atom can be transformed to a regular one.

\begin{lemma}\label{regular_incl_atm_lm}
Let $\mathsf{x}\subseteq \mathsf{a}$ be a primitive inclusion atom. Either $\mathsf{x}\subseteq \mathsf{a}\dashv\vdash\top$, or $\mathsf{x}\subseteq \mathsf{a}\dashv\vdash\bot$, or $\mathsf{x}\subseteq \mathsf{a}\dashv\vdash\mathsf{x}_0\subseteq \mathsf{a}_0$ for some regular primitive inclusion atom $\mathsf{x}_0\subseteq \mathsf{a}_0$.
\end{lemma}
\begin{proof}
We first eliminate constants $\top,\bot$ one by one from the right side of a primitive inclusion atom $\mathsf{x}\subseteq \mathsf{a}$. Consider a constant $v$ in $\mathsf{a}$. By \incexc we may without loss of generality assume that $v$ occurs at the last position of the sequence $\mathsf{a}$, i.e., the inclusion atom $\mathsf{x}\subseteq \mathsf{a}$ is $\mathsf{y}z\subseteq\mathsf{b}v$.  If $z=v$, 
by \incctr and \inczext, we have that 
\[\mathsf{y}\top\subseteq \mathsf{b}\top\dashv\vdash \mathsf{y}\subseteq \mathsf{b}\text{ and }\mathsf{y}\bot\subseteq \mathsf{b}\bot\dashv\vdash \mathsf{y}\subseteq \mathsf{b}.\]
For the special case when $\mathsf{y}$ and $\mathsf{b}$ are the empty sequence $\langle\rangle$, by \incid we have $\vdash \langle\rangle\subseteq \langle\rangle$. Then we derive that $\mathsf{y}\top\subseteq \mathsf{b}\top\dashv\vdash\top$ and $\mathsf{y}\bot\subseteq \mathsf{b}\bot\dashv\vdash\top$ by \topi.

If $z\neq v$, 
we show that 
\[\mathsf{y}\top\subseteq \mathsf{b}\bot\dashv\vdash\bot\text{ and }\mathsf{y}\bot\subseteq \mathsf{b}\top\dashv\vdash\bot.\] 
The right to left direction of the above two clauses follows from \Cref{der_rule_prop}(\ref{der_rule_exfalso}) (which is true also for the system of \PIncz). For the other direction, we only give the proof for $\mathsf{y}\top\subseteq \mathsf{b}\bot\vdash\bot$, the other case being symmetric. By \inczext, we have that $\mathsf{y}\top\subseteq \mathsf{b}\bot\vdash\mathsf{y}\top\bot\subseteq \mathsf{b}\bot\bot$. Since $\vdash\bot\leftrightarrow\bot$ by classical rules, we derive by $\subseteq\!\textsf{Cmp}$ that $\mathsf{y}\top\subseteq \mathsf{b}\bot\vdash\top\leftrightarrow\bot\vdash\bot$, as required.

Lastly, we remove repeated propositional variables from the right side of a primitive inclusion atom $\mathsf{x}\subseteq \mathsf{a}$. By \incctr and $\subseteq\textsf{Wk}$, we have that 
\[\mathsf{y}\top\top\subseteq \mathsf{b}pp\dashv\vdash \mathsf{y}\top\subseteq \mathsf{b}p~~\text{ and }~~\mathsf{y}\bot\bot\subseteq \mathsf{b}pp\dashv\vdash \mathsf{y}\bot\subseteq \mathsf{b}p.\]
For the last case, we show that
\(\mathsf{y}\top\bot\subseteq \mathsf{b}pp\dashv\vdash \bot.\)
The right to left direction follows from \Cref{der_rule_prop}(\ref{der_rule_exfalso}). For the other direction, since $\vdash p\leftrightarrow p$, by  $\subseteq\!\textsf{Cmp}$ we derive that $\mathsf{y}\top\bot\subseteq \mathsf{b}pp\vdash\top\leftrightarrow\bot\vdash\bot$.
\end{proof}

Finally, we are ready give the proof the normal form lemma, \Cref{DNF_PInc}.

\begin{proof}[Proof of \Cref{DNF_PInc}]
We  prove the lemma by induction on  $\phi$. 
If $\phi(p_{1},\dots,p_{n})=p_{i}$, then 
\begin{align*}
p_{i}&\dashv\vdash\mathop{\bigsor_{v\in 2^{\mathsf{N}\setminus\{p_i\}}}}(p_{1}^{v(1)}\wedge\dots\wedge p_{i-1}^{v(i-1)}\wedge p_{i}\wedge p_{i+1}^{v(i+1)}\wedge\dots\wedge p_{n}^{v(n)}) \\
&\dashv\vdash\mathop{\bigsor_{v\in 2^{\mathsf{N}\setminus\{p_i\}}}}(\Theta_{\{v\}}\wedge p_i\wedge \underline{v(1)}\dots\underline{v(i-1)}\top\underline{v(i+1)}\dots\underline{v(n)}\subseteq p_1\dots p_n)\tag{\Cref{der_rule_prop_pinc}(\ref{der_rule_pinc_winci})}\\
&\dashv\vdash\bigvee_{\{u\}\in \mathcal{X}_i}(\Theta_{\{u\}}\wedge\Phi_{\{u\}})\text{ where }\mathcal{X}_i=\{\{u\}\mid u\in 2^{\mathsf{N}},u(i)=1\}.
\end{align*}

 If $\phi(p_{1},\dots,p_{n})=\top$,  by \Cref{der_rule_prop_pinc}(\ref{der_rule_pinc_allposs}) we have  that
 \[\vdash \bigsor_{\{v\}\in \mathcal{X}_\top}(\Theta_{\{v\}}\wedge\Phi_{\{v\}}),,\text{ where }\mathcal{X}_\top=\{\{v\}\mid v\in 2^{\mathsf{N}}\}.\]
 Then, by \topi, we have $\top\dashv\vdash\bigsor_{\{v\}\in \mathcal{X}_\top}(\Theta_{\{v\}}\wedge\Phi_{\{v\}})$.
If $\phi(p_{1},\dots,p_{n})=\bot$, then trivially $\bot\dashv\vdash\bigsor\emptyset=\bot$.


If $\phi(p_{1},\dots,p_{n})$ is a  primitive inclusion atom. By \Cref{regular_incl_atm_lm} and \incexc, we may assume that $\phi=\top$ or $\phi=\bot$ or $\phi=x_1\dots x_k\subseteq p_1\dots p_k$ ($k\leq n$) is regular. The first two cases reduce to the previous cases. For the last case, we show that 
\(\phi\dashv\vdash\bigsor_{X\in \mathcal{X}}(\Theta_X\wedge\Phi_X)\)
where 
\[\mathcal{X}=\{X\subseteq 2^{\mathsf{N}}\mid \exists v\in X\text{ such that }v(p_1)=x_1,\dots,v(p_k)=x_k\}.\] 
For the right to left direction, by \sore it suffices to show that $\Theta_X,\Phi_X\vdash x_1\dots x_k\subseteq p_1\dots p_k$ for each $X\in \mathcal{X}$. For the valuation $ v\in X$ such that $v(p_1)=x_1,\dots,v(p_k)=x_k$, we know that 
 $x_1\dots x_k\underline{v(k+1)}\dots \underline{v(n)}\subseteq p_{1}\dots p_{k}p_{k+1}\dots p_n$
is a conjunct in $\Phi_X$. Thus, we derive $\Phi_X\vdash x_1\dots x_k\subseteq p_1\dots p_k$ by  \incctr. 

Conversely, for the left to right direction, we first have by \Cref{der_rule_prop_pinc}(\ref{der_rule_pinc_allposs}) that  $\vdash \bigsor_{v\in 2^{\mathsf{N}}}(\Theta_v\wedge \Phi_v)$. Then it suffices to derive $\bigsor_{v\in 2^{\mathsf{N}}}(\Theta_v\wedge \Phi_v),x_1\dots x_k\subseteq p_1\dots p_k\vdash\bigsor_{X\in \mathcal{X}}(\Theta_X\wedge\Phi_X)$, which by \Cref{veesub_eli}(\ref{veesub_eli_c}) reduces to derive that for each nonempty $Y\subseteq 2^{\mathsf{N}}$, 
\begin{equation}\label{DNF_PInc_eq1}
\Theta_Y, \Phi_Y,x_1\dots x_k\subseteq p_1\dots p_k\vdash \bigsor_{X\in \mathcal{X}}(\Theta_X\wedge\Phi_X).
\end{equation}
Now, if $Y\in \mathcal{X}$, then the above holds by \sori. Otherwise, if $Y\notin \mathcal{X}$, then for each $v\in Y$, $v(p_i)\neq x_i$ for some $1\leq i\leq k$. Thus, 
\[\Theta_v\vdash p_1^{v(1)}\wedge\dots\wedge p_k^{v(k)}\vdash \neg ( p_1^{x_1}\wedge\dots\wedge p_k^{x_k}),\]
which implies $\Theta_Y\vdash \neg ( p_1^{x_1}\wedge\dots\wedge p_k^{x_k})$ by \sore. By \Cref{der_rule_prop_pinc}(\ref{der_rule_pinc_ince}), we have that $\neg ( p_1^{x_1}\wedge\dots\wedge p_k^{x_k}),x_1\dots x_k\subseteq p_1\dots p_k\vdash\bot$.
Hence, we obtain (\ref{DNF_PInc_eq1}) by \Cref{der_rule_prop}(\ref{der_rule_exfalso}).



\vspace{0.5\baselineskip}

Suppose $\alpha$ is a classical formula, and $\alpha\dashv\vdash\bigsor_{X\in\mathcal{X}}(\Theta_{X}\wedge\Phi_X)$.  We show that $\neg\alpha\dashv\vdash\bigsor_{v\in 2^{\mathsf{N}}\setminus\bigcup\mathcal{X}}(\Theta_{\{v\}}\wedge\Phi_{\{v\}})$. It is sufficient to prove that $\bigsor_{X\in\mathcal{X}}(\Theta_{X}\wedge\Phi_X)\dashv\vdash\bigsor_{v\in \bigcup\mathcal{X}}\Theta_{\{v\}}$, since we will then have that
\begin{align*}
\neg\alpha&\dashv\vdash\neg\bigsor_{v\in \bigcup\mathcal{X}}\Theta_{\{v\}}\tag{by the standard rules for classical formulas}\\
&\dashv\vdash\bigsor_{v\in 2^{\mathsf{N}}\setminus\bigcup\mathcal{X}}\Theta_{\{v\}}\tag{by the standard rules for classical formulas}\\
&\dashv\vdash\bigsor_{v\in 2^{\mathsf{N}}\setminus\bigcup\mathcal{X}}(\Theta_{\{v\}}\wedge\Phi_{\{v\}})\tag{\Cref{der_rule_prop_pinc}(\ref{der_rule_pinc_winci}) and \ce}
\end{align*}
 Now, by the same argument as that in the proof of Lemma \ref{DNF_PU}, since $\alpha$ is flat, we have that for each $v\in \bigcup\mathcal{X}$, $\{v\}\in \mathcal{X}$.
 Moreover, by \Cref{der_rule_prop_pinc}(\ref{der_rule_pinc_winci}), $\Theta_{\{v\}}\vdash \Theta_{\{v\}}\wedge\Phi_{\{v\}}$. Thus, the direction $\bigsor_{v\in \bigcup\mathcal{X}}\Theta_{\{v\}}\vdash\bigsor_{X\in\mathcal{X}}(\Theta_{X}\wedge\Phi_X)$ follows from \sori and \sore. For the other direction,  for each $X\in\mathcal{X}$, we derive by \sori  that
\[
\Theta_{X}\wedge\Phi_X\vdash \bigsor_{u\in X}\Theta_{\{u\}}\vdash\bigsor_{v\in \bigcup\mathcal{X}}\Theta_{\{v\}},
\]
from which $\bigsor_{X\in\mathcal{X}}(\Theta_{X}\wedge\Phi_X)\vdash\bigsor_{v\in \bigcup\mathcal{X}}\Theta_{\{v\}}$ follows by $\sor$\textsf{E}.

\vspace{0.5\baselineskip}

Suppose $\psi(\mathsf{N})$ and $\chi(\mathsf{N})$ satisfy  
  \begin{equation}\label{pu_nf_proof_IH}
  \psi\dashv\vdash\bigsor_{X\in\mathcal{X}}(\Theta_{X}\wedge\Phi_X)\quad\text{and}\quad\chi\dashv\vdash\bigsor_{Y\in\mathcal{Y}}(\Theta_{Y}\wedge\Phi_Y),
  \end{equation}
for some  finite sets $\mathcal{X}$ and $\mathcal{Y}$ of $\mathsf{N}$-teams. The case $\phi=\psi\sor\chi$ follows from induction hypothesis. If $\phi=\psi\wedge\chi$, we show that $\psi\wedge\chi\dashv\vdash\bigsor_{Z\in \mathcal{Z}}(\Theta_Z\wedge\Phi_Z)$, where
\[\mathcal{Z}=\{\bigcup \mathcal{X}'\mid \mathcal{X}'\subseteq \mathcal{X}\text{ and }\bigcup \mathcal{X}'=\bigcup \mathcal{Y}'\text{ for some }\mathcal{Y}'\subseteq \mathcal{Y}\}.\]
For the right to left direction, by \sore it suffices to derive $\Theta_Z,\Phi_Z\vdash\psi\wedge\chi$ for each $Z=\bigcup \mathcal{X}'=\bigcup \mathcal{Y}'\in\mathcal{Z}$, where $\mathcal{X}'\subseteq \mathcal{X}$ and $\mathcal{Y}'\subseteq \mathcal{Y}$.  By \Cref{PInc_dstr_lm} and \sori, we have that $\Theta_Z,\Phi_Z\vdash\bigvee_{X\in \mathcal{X}'}(\Theta_X\wedge\Phi_X)\vdash\bigvee_{X\in \mathcal{X}}(\Theta_X\wedge\Phi_X)\vdash\psi$. Similarly, $\Theta_Z,\Phi_Z\vdash\chi$.

For the left to right direction, by \Cref{veesub_eli}(\ref{veesub_eli_c}) it suffices to prove that for each nonempty $\mathcal{X}'\subseteq \mathcal{X}$ and $\mathcal{Y}'\subseteq \mathcal{Y}$,
\[
\bigsor_{X\in \mathcal{X}'}\Theta_X,\bigwedge_{X\in \mathcal{X}'}\Phi_X,\bigsor_{Y\in \mathcal{Y}'}\Theta_Y,\bigwedge_{Y\in \mathcal{Y}'}\Phi_Y\vdash\bigsor_{Z\in\mathcal{Z}}(\Theta_Z\wedge\Phi_Z),
\]
which reduces to showing that 
\begin{equation}\label{DNF_PInc_eq2}
\Theta_{\bigcup\mathcal{X}'},\Theta_{\bigcup\mathcal{X}'},\Theta_{\bigcup\mathcal{Y}'},\Phi_{\bigcup\mathcal{Y}'}\vdash\bigsor_{Z\in\mathcal{Z}}(\Theta_Z\wedge\Phi_Z)
\end{equation}
as elements in $\mathcal{X}'$ and $\mathcal{Y}'$ may not be disjoint.
Now, if $\bigcup \mathcal{X}'=\bigcup \mathcal{Y}'\in \mathcal{Z}$, then the above clause follows easily from \sori. Otherwise, if $\bigcup \mathcal{X}'\neq\bigcup \mathcal{Y}'$, 
assume w.l.o.g. there exists some $v\in \bigcup \mathcal{X}'\setminus \bigcup \mathcal{Y}'$. First, we derive 
\(\Phi_v,\neg\Theta_{v}\vdash\bigsor_{Z\in\mathcal{Z}}(\Theta_Z\wedge\Phi_Z)\) by \Cref{der_rule_prop_pinc}(\ref{der_rule_pinc_ince}).
Since $v\notin \bigcup \mathcal{Y}'$, by the standard rules for classical formulas, we have that  $\Theta_{\bigcup\mathcal{Y}'}\vdash \neg\Theta_{v}$. Putting these together, we obtain (\ref{DNF_PInc_eq2}).
\end{proof}

\subsection{\PInc}

In this section, we extend the system of \PIncz to obtain a sound and complete system for propositional inclusion logic \PInc with arbitrary inclusion atoms. Recall from \Cref{inc_2prim}(\ref{inc_2prim_1}) that arbitrary inclusion atoms are definable in terms of primitive ones. 
Such an interaction between arbitrary and primitive inclusion atoms are characterized by the two rules we add to the system of \PInc, inclusion atom extension \incext  and   reduction \incrdt rule. 


\begin{definition}
The system of \PInc consists of all rules in the system of \PIncz together with the rules in \Cref{rules_PInc}, where $\mathsf{a}$ and $\mathsf{b}$ are arbitrary (and possibly empty) sequences of elements in $\textsf{Prop}\cup\{\top,\bot\}$,  $\mathsf{x}$ stands for an arbitrary sequence of constants from $\{\top,\bot\}$, and $|\mathsf{a}|$ denotes the length of the sequence $\mathsf{a}$.
\end{definition}

\begin{table}[t]
\begin{center}
\caption{Rules for inclusion atoms}
\vspace{4pt}
\setlength{\tabcolsep}{6pt}
\renewcommand{\arraystretch}{1.8}
\setlength{\extrarowheight}{1pt}
\scalebox{.92}{\begin{tabular}{|C{0.46\linewidth}C{0.46\linewidth}|}
\hline
\AxiomC{}\noLine\UnaryInfC{{\small$D$}}\noLine\UnaryInfC{$\displaystyle\bigwedge_{\mathsf{x}\in\{\top,\bot\}^{|\mathsf{a}|}}\big(\mathsf{a}^{\mathsf{x}}\to\mathsf{x}\subseteq \mathsf{b}\big)$}\RightLabel{\incext} \UnaryInfC{$\mathsf{a}\subseteq \mathsf{b}$}\noLine\UnaryInfC{}   \DisplayProof
&\AxiomC{{\small$D$}}\noLine\UnaryInfC{$\mathsf{a}\subseteq \mathsf{b}$}\RightLabel{\incrdt} \UnaryInfC{$\mathsf{a}^{\mathsf{x}}\to\mathsf{x}\subseteq \mathsf{b}$}\noLine\UnaryInfC{}   \DisplayProof\\
\hline
\end{tabular}
}
\label{rules_PInc}
\vspace{-26pt}
\end{center}
\end{table}


By \Cref{inc_2prim}(\ref{inc_2prim_1}) the two new rules \incext and \incrdt are clearly sound, and thus the system is sound.
By applying \incext and \incrdt, we can easily reduce an arbitrary inclusion atom to a formula with primitive inclusion atoms only:
\[\mathsf{a}\subseteq \mathsf{b}\dashv\vdash \bigwedge_{\mathsf{x}\in\{\top,\bot\}^{|\mathsf{a}|}}\big(\mathsf{a}^{\mathsf{x}}\to\mathsf{x}\subseteq \mathsf{b}\big).\]
From this the completeness of the system of \PInc follows.

\begin{theorem}[Completeness]
For any set $\Gamma\cup\{\phi\}$ of \PInc-formulas, we have that $\Gamma\models\phi \iff \Gamma\vdash\phi$.  
\end{theorem}

Let us end this section by illustrating the derivation of the replacement rule for inclusion atoms in the system of \PInc.


\begin{example}
\begin{enumerate}[(i)]
\item $a\leftrightarrow b, a\mathsf{c}\subseteq d\mathsf{e}\vdash b\mathsf{c}\subseteq d\mathsf{e}$.
\item $a\leftrightarrow b,d\mathsf{e}\subseteq a\mathsf{c}\vdash d\mathsf{e}\subseteq b\mathsf{c}$.
\end{enumerate}
\end{example}
\begin{proof}
(i) By \incext it suffices to show  $a\leftrightarrow b, a\mathsf{c}\subseteq d\mathsf{e}\vdash b\mathsf{c}^{x\mathsf{y}}\to x\mathsf{y}\subseteq d\mathsf{e}$ for all $x\mathsf{y}\in\{\top,\bot\}^{|\mathsf{y}|+1}$. First, by \incrdt we have $a\mathsf{c}\subseteq d\mathsf{e}\vdash a\mathsf{c}^{x\mathsf{y}}\to x\mathsf{y}\subseteq d\mathsf{e}$. Next, since 
\[a\leftrightarrow b,\neg a\mathsf{c}^{x\mathsf{y}}\vdash \neg b\mathsf{c}^{x\mathsf{y}}\vdash  \neg b\mathsf{c}^{x\mathsf{y}}\vee x\mathsf{y}\subseteq d\mathsf{e}\]
and $x\mathsf{y}\subseteq d\mathsf{e}\vdash  \neg b\mathsf{c}^{x\mathsf{y}}\vee x\mathsf{y}\subseteq d\mathsf{e}$, we obtain by applying \sore that 
\[a\leftrightarrow b, \neg a\mathsf{c}^{x\mathsf{y}}\vee x\mathsf{y}\subseteq d\mathsf{e}\vdash  \neg b\mathsf{c}^{x\mathsf{y}}\vee x\mathsf{c}\subseteq d\mathsf{e}\] 
i.e., $a\leftrightarrow b,  a\mathsf{c}^{x\mathsf{y}}\to x\mathsf{y}\subseteq d\mathsf{e}\vdash   b\mathsf{c}^{x\mathsf{y}}\to x\mathsf{c}\subseteq d\mathsf{e}$.
Hence, we conclude that $a\leftrightarrow b, a\mathsf{c}\subseteq d\mathsf{e}\vdash b\mathsf{c}^{x\mathsf{y}}\to x\mathsf{c}\subseteq d\mathsf{e}$ as required.

(ii). By item (i) we derive that $b\mathsf{c}\subseteq b\mathsf{c},a\leftrightarrow b\vdash a\mathsf{c}\subseteq b\mathsf{c}$, which implies $a\leftrightarrow b\vdash a\mathsf{c}\subseteq b\mathsf{c}$, since $\vdash b\mathsf{c}\subseteq b\mathsf{c}$ by \incid. By \inctrs, we have that $d\mathsf{e}\subseteq a\mathsf{c},a\mathsf{c}\subseteq b\mathsf{c}\vdash d\mathsf{e}\subseteq b\mathsf{c}$. Thus, we conclude that $a\leftrightarrow b,d\mathsf{e}\subseteq a\mathsf{c}\vdash d\mathsf{e}\subseteq b\mathsf{c}$, as required.
%
\end{proof}

\section{Locality revisited and interpolation}

Having studied the expressive power and proof theory for our union closed team logics in the previous two sections, let us in this section revisit the fundamental property of these logics, the {\em locality property}. 
%
%
The locality property states that the propositional variables not occurring in a formula are irrelevant for the evaluation of the formula (see \Cref{empty_prop_locality_proof}).  
It was observed already in \cite{Pietro_I/E} in the context of first-order inclusion logic that such a basic property is in fact 
nontrivial in the team-based logics (especially in non-downwards closed team logics) and thus cannot be taken for granted. In particular, a counterexample was given in \cite{Pietro_I/E} to show that first-order inclusion logic under the so-called {\em strict semantics} does not any more satisfy locality. 
In this section, we give counterexamples to show that locality fails for propositional inclusion logic \PInc and \PUw and \PAm as well if strict semantics is applied.


We also point out a subtle connection between locality and interpolation property, 
where the interpolation property states that for any entailment $\phi\models\psi$, there exists an interpolant $\theta$ in the common language of $\phi$ and $\psi$.
It follows from  recent work \cite{DAgostino19} by D'Agostino in the modal team logics setting that  all of the expressively complete propositional team logics  admit uniform interpolation (a stronger property than Craig's interpolation requiring, in addition, that the interpolants are uniform). Therefore, the expressively complete union closed logics we consider in this paper (e.g., \PUw, \PInc and \PAm, see \Cref{PU_excmp,PT_exp_pw}) all admit uniform interpolation, demonstrating also that our logics have good meta-logical properties. 
The proof in \cite{DAgostino19} requires the assumption that the logics in question satisfy the locality property.
We reformulate and elaborate this proof  in our setting, so as to also highlight the crucial role that the locality property plays in this argument. It is thus natural to ask whether locality is actually a necessary condition or presupposition for interpolation. Roughly speaking,  the locality property and the interpolation property both describe the redundant role of irrelevant variables. This similar flavor seems to suggest that these two properties may actually be connected. This problem is particularly relevant for union closed team logics, as locality may fail for these logics  under the strict semantics. In line with this discussion, we give, in this section, an example of a non-local fragment of \PUw under strict semantics in which 
Craig's interpolation actually  fails. 

As described above, the non-triviality of the locality property in the team semantics setting is (at least partly) due to the fact that there are alternative versions of the semantics. The team semantics we defined in \Cref{sec:pre} is known 
as the {\em lax semantics}. Another competing version is called the {\em strict semantics}, which is otherwise the same as lax semantics except that the semantics for the disjunction is subtly different. 
%
The strict semantics for the disjunctions $\vee$ and $\ror$  require the team in question to be split into two disjoint subteams. More precisely,  the satisfaction relation $\models^s$ for the disjunctions in strict semantics is defined as:
\begin{itemize}
\item $X\models^s\phi\sor\psi$ ~iff~ there exist $Y,Z\subseteq X$ such that $Y\cap Z=\emptyset$, $X=Y\cup Z$, $Y\models^s\phi$ and $Z\models^s\psi$. 
\item $X\models^s\phi\ror\psi$ ~~iff~~ $X=\emptyset$ or there exist nonempty subteams $Y,Z\subseteq X$ such that $Y\cap Z=\emptyset$, $X=Y\cup Z$, $Y\models\phi$ and $Z\models\psi$.
\end{itemize}
Clearly, the strict semantics for $\vee$ coincide with the lax semantics in downwards closed logics. The union closed logics we consider in this paper are, however, not downwards closed.

We now present our examples to illustrate that under strict semantics none of the  logics \PUw, \PInc and \PAm satisfies the locality property. 


\begin{example}\label{non-local-exm}
Consider the team $X$  over domain $\{p,q,r,s\}$ illustrated  in the left table below:
{\small\begin{center}
$X$: ~~\begin{tabular}{|c|c|c|c|}
\hline
 $p$ &$q$&$r$&$s$
  \\\hline
\cellcolor[gray]{.7}$1$&$0$\cellcolor[gray]{.7}&\cellcolor[gray]{.7}$0$&$0$\\\hline
\cellcolor[gray]{.7}$0$&\cellcolor[gray]{.7}$1$ &\cellcolor[gray]{.7}$0$&$0$\\\hline
\cellcolor[gray]{.9}$0$ &$1$\cellcolor[gray]{.9}&\cellcolor[gray]{.9}$0$&$1$\\\hline
\cellcolor[gray]{.9}$0$ &$0$\cellcolor[gray]{.9}&\cellcolor[gray]{.9}$1$&$0$\\\hline
 \end{tabular}
\quad\quad\quad$X\upharpoonright\{p,q,r\}$: ~~\begin{tabular}{|c|c|c|c|}
\hline
 &$p$ &$q$&$r$
  \\\hline
 $v_1$& \cellcolor[gray]{.7}$1$&$0$\cellcolor[gray]{.7}&\cellcolor[gray]{.7}$0$\\\hline
 $v_2$& \cellcolor[gray]{.7}$0$ &\cellcolor[gray]{.7}$1$&\cellcolor[gray]{.7}$0$\\\hline
 $v_3$& \cellcolor[gray]{.9}$0$ &\cellcolor[gray]{.9}$0$&\cellcolor[gray]{.9}$1$\\\hline
 \end{tabular}
\end{center}
}
Under strict semantics the \PU-formula $(p\ror q)\vee(q\ror r)$ and the \PAm-formula $(\neg p\wedge\anm(q))\vee(\anm(q)\wedge \neg r)$ are both satisfied by $X$, because $X$ can be split into two disjoint subteams (illustrated by two shaded table fragments of different tones) each satisfying one distinct $\vee$-disjunct from each formula. But both formulas fail in the restricted team $X'=X\upharpoonright\{p,q,r\}$ under strict semantics, because the full team $X'$ does not satisfy any of the four $\vee$-disjuncts in the two formulas, and $\{v_1,v_2\}$ is the only nonempty subteam of $X'$ that satisfies $p\ror q$ and $\anm(q)\wedge \neg r$, leaving the reminder subteam $\{v_3\}$ falsifying $q\ror r$ and $\neg p\wedge\anm(q)$.

Consider also the team $Y$ over domain $\{p,q,r,s,t,u,v\}$ defined below: 
{\small\begin{center}
$Y$: ~~\begin{tabular}{|c|c|c|c|c|c|c|}
\hline
 $p$ &$q$&$r$&$s$&$t$& $u$ & $v$
  \\\hline
\cellcolor[gray]{.7}$0$&$0$\cellcolor[gray]{.7}&$1$\cellcolor[gray]{.7}&$1$\cellcolor[gray]{.7}&$0$& $1$  &$0$\\\hline
\cellcolor[gray]{.7}$1$&\cellcolor[gray]{.7}$1$ &\cellcolor[gray]{.7}$0$&\cellcolor[gray]{.7}$0$&$1$&$0$ &$1$\\\hline
$1$ &$1$&\cellcolor[gray]{.9}$0$&\cellcolor[gray]{.9}$0$&\cellcolor[gray]{.9}$1$&\cellcolor[gray]{.9}$0$&$0$\\\hline
$0$ &$1$&\cellcolor[gray]{.9}$1$&\cellcolor[gray]{.9}$0$&\cellcolor[gray]{.9}$0$&\cellcolor[gray]{.9}$0$ &$0$\\\hline
 \end{tabular}
\quad\quad\quad$Y'$: ~~\begin{tabular}{|c|c|c|c|c|c|c|}
\hline
& $p$ &$q$&$r$&$s$&$t$& $u$ 
  \\\hline
 $w_1$&  \cellcolor[gray]{.7}$0$&$0$\cellcolor[gray]{.7}&$1$\cellcolor[gray]{.7}&$1$\cellcolor[gray]{.7}&$0$& $1$\\\hline
 $w_2$&\cellcolor[gray]{.7}$1$ &\cellcolor[gray]{.7}$1$&\cellcolor[gray]{.7}$0$&\cellcolor[gray]{.7}$0$&$1$&$0$\\\hline
 $w_3$&$0$ &$1$&\cellcolor[gray]{.9}$1$&\cellcolor[gray]{.9}$0$&\cellcolor[gray]{.9}$0$&\cellcolor[gray]{.9}$0$ \\\hline
 \end{tabular}
\end{center}
}
Under strict semantics the \PInc-formula $pq\subseteq rs\sor tu\subseteq rs$ is satisfied by the team $Y$, but falsified by the restricted team $Y'=Y\upharpoonright\{p,q,r,s,t,u\}$.
%
%
%
%
\end{example}

Let us remark that the above example actually only shows that formulas in the logics \PU, \PAm and \PInc under strict semantics are not {\em downwards} local, where we say that a formula $\phi(\mathsf{N})$ is {\em downwards local} if  for any teams $X$ and $Y$ with $\textsf{dom}(X)\supseteq \textsf{dom}(Y)\supseteq \mathsf{N}$ and $X\upharpoonright \mathsf{N}=Y\upharpoonright \mathsf{N}$, it holds that
\[X\models\phi\Longrightarrow Y\models\phi.\]
It is easy to verify that a formula $\phi$ is local iff it is both downwards and upwards local, where we say that $\phi(\mathsf{N})$ is {\em upwards local} if for any $X$ and $Y$ as above,
\[X\models\phi\Longrightarrow Y\models\phi.\]
By a straightforward inductive argument (similar to the proof of \Cref{empty_prop_locality_proof}), one can prove that the logics \PU, \PAm and \PInc under strict semantics are nevertheless upwards closed. 

\begin{remark}\label{rmk_union}
Consider again the teams $X,Y$ and the three formulas in Example \ref{non-local-exm}. It is easy to see that under strict semantics, the first two formulas are satisfied in both $\{v_1,v_2\}$ and $\{v_2,v_3\}$ but not in their union $X'$; similarly, the third formula is satisfied in both $\{w_1,w_2\}$ and $\{w_2,w_3\}$ but not in their union $Y'$. This shows that none of the three logics \PU, \PInc and \PAm is any more union closed when strict semantics is applied. This fact for propositional inclusion logic \PInc was observed already  in \cite{HellaKuusistoMeierVirtema17}. \PInc  behaves differently under strict and lax semantics also in terms of computational properties; the reader is referred to \cite{HellaKuusistoMeierVirtema17,HellaKuusistoMeierVollmer19} for details.
%
\end{remark}

\begin{remark}
 Given any union closed $\mathsf{N}$-team property $\mathsf{P}\in \mathbb{P}^{\dot{\overline{\cup}}}$ that contains the empty team. It is not hard to see that the \PU-formula $\bigvee_{X\in\mathsf{P}}\Psi_X$ under strict semantics still characterizes $\mathsf{P}$. The corresponding formulas in \PInc and \PAm  defined in \Cref{PT_exp_pw} are still equivalent to $\bigvee_{X\in\mathsf{P}}\Psi_X$ under strict semantics. Therefore under strict semantics all properties in $\mathbb{P}^{\dot{\overline{\cup}}}$ are still definable in the three logics \PU, \PInc and \PAm. But as illustrated in \Cref{rmk_union}, under strict semantics these three logics can also define properties that are not union closed. Determining the expressive power of these logics under strict semantics 
 is left as future work. It is worthwhile to mention that  first-order inclusion logic is known to have the same expressive power as positive greatest fixed point logic under lax semantics \cite{inclusion_logic_GH}, whereas under strict semantics it is  
 so strictly stronger that it is equivalent to existential second-order logic \cite{Hierarchies_Ind_GHK}. 
\end{remark}

Let us now define another key notion for this section, namely the (uniform) interpolation property. 
\begin{definition}
We say that a logic \LL enjoys {\em (Craig's) interpolation property} if for any pair of \LL-formulas $\phi(\mathsf{KN})$ and  $\psi(\mathsf{MN})$ with  $\mathsf{K},\mathsf{M},\mathsf{N}$ pairwise disjoint sets of propositional variables, if $\phi\models\psi$, then there exists an \LL-formula $\theta(\mathsf{N})$ in the common language $\mathsf{N}$ (called an {\em interpolant}) such that
\(\phi\models\theta\text{ and }\theta\models\psi.\)
\end{definition}
Depending on the formula $\psi$, the interpolant $\theta$ may  be different. Uniform interpolation property requires the interpolant to be  uniform for all such $\psi$. 
\begin{definition}
We say that a  logic \LL enjoys {\em uniform interpolation property} if for any \LL-formula $\phi(\mathsf{K})$ and any $\mathsf{N}\subseteq \mathsf{K}$, there is an \LL-formula $\theta(\mathsf{N})$ (called a {\em uniform interpolant})  such that $\phi\models\theta$, and for any \LL-formula $\psi(\mathsf{M})$ with $\mathsf{K}\cap\mathsf{M}\subseteq \mathsf{N}$, we have that
$\phi\models\psi$ implies $\theta\models\psi$.
\end{definition}

Clearly, uniform interpolation implies Craig's interpolation. For more in-depth discussions on interpolation, the reader is referred to, e.g., \cite{GabbayMaksimova2005,Hoogland2001}. We now proceed to reformulate the result in \cite{DAgostino19} that given the locality property, any propositional team-based logic that is expressively complete in some forgetful class of team properties enjoys uniform interpolation, and thus all of the expressively complete union closed team logics from Theorems \ref{PU_excmp} and \ref{PT_exp_pw} enjoy uniform interpolation. We call a class $\mathbb{P}$ of team properties  \emph{forgetful} if for any $\mathsf{N}$-team property $\mathsf{P}\in\mathbb{P}$ and any $\mathsf{M}\subseteq \mathsf{N}$, $\mathsf{P}\upharpoonright_{\mathsf{M}}\in \mathbb{P}$, where $\mathsf{P}\upharpoonright_{\mathsf{M}}=\{X\upharpoonright \mathsf{M}: X\in \mathsf{P}\}$. Intuitively, the notion is such termed as the team property $\mathsf{P}\upharpoonright_{\mathsf{M}}$ in the definition simply ``forgets" the information concerning all propositional variables in the set $\mathsf{N}\setminus \mathsf{M}$.
For example, the collection $\mathbb{P}^{\dot{\overline{\cup}}}$ of all union closed team properties which contain the empty team is forgetful, so are  
the collection $\mathbb{F}$ of all flat team properties, and the collection of all downwards closed team properties which contain the empty team (in which propositional dependence logic is expressively complete \cite{VY_PD}), etc.

One important lemma in the argument of \cite{DAgostino19} (formulated in our setting) is the observation that team semantics has the {\em amalgamation property} in the following sense, where we write simply $\mathsf{MN}$ for the union $\mathsf{M}\cup\mathsf{N}$ of two domains $\mathsf{M},\mathsf{N}$.

\begin{lemma}[Amalgamation]\label{amag}
For any $\mathsf{K}$-team $X$ and $\mathsf{M}$-team $Y$ such that  $X\upharpoonright(\mathsf{K}\cap \mathsf{M})=Y\upharpoonright (\mathsf{K}\cap \mathsf{M})$, there exists a $\mathsf{KM}$-team $Z$ such that $Z\upharpoonright \mathsf{K}=X$ and $Z\upharpoonright \mathsf{M}=Y$.
\end{lemma}
\begin{proof}
Clearly the required $\mathsf{KM}$-team $Z$ can be defined as
\[Z=\{v:\mathsf{KM}\cup\{\top,\bot\}\to\{0,1\}\mid v\upharpoonright\mathsf{K}\in X,~v\upharpoonright\mathsf{M}\in Y\text{ and } v(\top\bot)=10\}.\]
\end{proof}

Now, we are ready to give the proof of the uniform interpolation result for expressively complete propositional team logics that satisfy locality property. The argument is due to \cite{DAgostino19}. We provide here a detailed proof in which all the steps involving applications of the (upwards and downwards) locality property are explicitly spelled out.



\begin{theorem}[Uniform interpolation]\label{uni_interpo}
Let \LL be a team-based propositional logic that has the locality property and is expressively complete in some forgetful class $\mathbb{P}$ of team properties. Then \LL enjoys uniform interpolation property. 
%
%
%
\end{theorem}

\begin{proof}
Let $\phi(\mathsf{K})$ be an \LL-formula and $\mathsf{N}\subseteq \mathsf{K}$.
Since \LL is expressively complete in $\mathbb{P}$, $\llbracket \phi\rrbracket_{\mathsf{K}}\in\mathbb{P}$. As $\mathbb{P}$ is forgetful, $\llbracket \phi\rrbracket_{\mathsf{K}}\!\!\upharpoonright_{\mathsf{N}}\in \mathbb{P}$ as well.
By the expressive completeness again, we  find an \LL-formula $\theta(\mathsf{N})$ such that  $\llbracket \theta\rrbracket_{\mathsf{N}}=\llbracket \phi\rrbracket_{\mathsf{K}}\!\!\upharpoonright_{\mathsf{N}}$. We show that $\theta$ is the required uniform interpolant.
 
 To see that $\phi\models\theta$, suppose $X\models\phi$ with $\textsf{dom}(X)\supseteq \mathsf{K}$. Since \LL is (downwards) 
 local, we have also that $X \upharpoonright \mathsf{K}\models\phi(\mathsf{K})$, i.e., $X\upharpoonright \mathsf{K}\in \llbracket\phi\rrbracket_{\mathsf{K}}$. Thus 
 \[X\upharpoonright \mathsf{N}=(X \upharpoonright \mathsf{K})\upharpoonright \mathsf{N}\in \llbracket \phi\rrbracket_{\mathsf{K}}\!\!\upharpoonright_{\mathsf{N}}=\llbracket \theta\rrbracket_{\mathsf{N}},\]
namely $X\upharpoonright \mathsf{N}\models\theta$. Then, since \LL is (upwards)
local, we conclude that $X\models\theta$.
 
 Next, assuming  that $\psi(\mathsf{M})$ is an \LL-formula with $\mathsf{K}\cap \mathsf{M}\subseteq \mathsf{N}$ and $\phi\models\psi$, we show that $\theta\models\psi$. Suppose $X\models\theta$ and $\textsf{dom}(X)\supseteq \mathsf{MN}$. Since \LL is (downwards)
 local, we have  $X\upharpoonright \mathsf{N}\models\theta(\mathsf{N})$. Thus 
 \(X\upharpoonright \mathsf{N}\in \llbracket \theta\rrbracket_{\mathsf{N}}=\llbracket \phi\rrbracket_{\mathsf{K}}\!\!\upharpoonright_{\mathsf{N}}.\)
 It follows that there exists a ${\mathsf{K}}$-team $Y$ such that $Y\models\phi$ and $Y\upharpoonright \mathsf{N}=X\upharpoonright \mathsf{N}$.
 
Since $\mathsf{N}\subseteq\mathsf{K}$ and $\mathsf{K}\cap\mathsf{M}\subseteq \mathsf{N}$, we have that
\[\mathsf{N}\subseteq\mathsf{K}\cap \mathsf{MN}=(\mathsf{K}\cap \mathsf{M})\cup(\mathsf{K}\cap\mathsf{N})\subseteq \mathsf{N},\]
thus $\mathsf{K}\cap \mathsf{MN}=\mathsf{N}$.
 Now, by \Cref{amag}, 
 there exists a $\mathsf{KMN}$-team $Z$ such that $Z\upharpoonright \mathsf{K}= Y$ and $Z\upharpoonright \mathsf{MN}= X\upharpoonright \mathsf{MN}$.
  Since $Y\models\phi(\mathsf{K})$, we have $Z\upharpoonright \mathsf{K}\models\phi$, which by the  (upwards)
 locality of \LL implies $Z\models\phi$. It then follows from assumption $\phi\models\psi$ that $Z\models\psi$. Since \LL is (downwards) local, 
 we obtain $Z\upharpoonright \mathsf{MN}\models\psi$, which implies $X\upharpoonright \mathsf{MN}\models\psi$. Hence, we conclude $X\models\psi$, as \LL is (upwards) 
 local.
\end{proof}


It then follows from \Cref{PU_excmp,PT_exp_pw} that uniform interpolation holds for all the expressively complete union closed team logics we consider in this paper.

\begin{corollary}
The  logics $\PUw, \PInc, \PIncz,\PAm$, and $\Picnst$ enjoy uniform interpolation property and thus also Craig's interpolation property.
\end{corollary}

Let us emphasize again that the proof of \Cref{uni_interpo} makes essential use of the locality property (both  upwards and downwards locality, to be more precise). 
It is not clear whether the locality property is actually a necessary condition for uniform interpolation. Yet let us now demonstrate 
that the interpolation property can fail for team-based logics without the locality property. Recall from \Cref{non-local-exm} that \PUw with strict semantics is not local, and the counterexample can be built with four propositional variables. We shall 
consider  the restricted language of \PUw with four propositional variables  $p,q,r,s$ and constants $\top,\bot$ only. This language,  denoted as $\PUw_4$, is clearly still not local under strict semantics. We now illustrate that $\PUw_4$ does not admit (Craig's) interpolation. 


\begin{example}\label{fail_interp_exam}
Consider $\PUw_4$ with strict semantics and consider the team $X$ from \Cref{non-local-exm} again. We claim that
\begin{equation}\label{non_interpolation_exm_eq1}
\Psi_{X'}\wedge \big((p\ror q)\vee(q\ror r)\big)\models^s s\ror \neg s,
\end{equation}
where $\Psi_{X'}=(p\wedge \neg q\wedge \neg r)\ror(\neg p\wedge q\wedge\neg r)\ror(\neg p\wedge\neg q\wedge r)$ is the formula that  defines (under lax semantics) the team $X'=X\upharpoonright\{p,q,r\}$ modulo the empty team in the sense of Equation (\ref{eq_psi_x_prop}) from \Cref{sec:expr_comp}. Now, observe that in the entailment (\ref{non_interpolation_exm_eq1}) the common language of two formulas on two sides of the turnstile ($\models$) is empty. Hence there is no interpolant for the entailment (\ref{non_interpolation_exm_eq1}) (the constants  $\bot$ and $\top$ are clearly not interpolants).

To see why (\ref{non_interpolation_exm_eq1}) holds, take any nonempty team $Y$ over the domain $\{p,q,r,s\}$ that satisfies both $\Psi_{X'}$ and $(p\ror q)\vee(q\ror r)$ under strict semantics. It is easy to see that $Y\models^s \Psi_{X'}$ implies that $Y\upharpoonright \{p,q,r\}=X'$. Now, similarly to what we have argued in \Cref{non-local-exm}, in order for the $\{p,q,r,s\}$-team $Y$ to satisfy $(p\ror q)\vee(q\ror r)$ under strict semantics,  the valuation $v_2$ in $X'$ must extend in $Y$ to two distinct valuations $v_2'$ and $v_2''$. In the language $\PUw_4$ with four propositional variables only, this can only be the case if $v_2'(s)=0$ and $v_2''(s)=1$. From this we must conclude that $Y\models^s s\ror \neg s$, as desired.
\end{example}


\section{Conclusion and further directions}

In this paper, we have studied the expressive power, axiomatization problem and locality property for several propositional union closed team logics. Building on the result in \cite{YangVaananen:17PT} that \PU is expressively complete, we proved that \PInc and \PAm as well as their  fragments  \PIncz and \Picnst are also expressively complete. 
It is interesting to note that our version of the propositional inclusion logic \PInc allows inclusion atoms $\mathsf{a}\subseteq\mathsf{b}$ with the constants $\top$ and $\bot$ in the arguments. As we illustrated, the original version of \PInc, the version in which inclusion atoms $\mathsf{p}\subseteq\mathsf{q}$ can only have propositional variables in the arguments, is actually strictly less expressive, and  not expressively complete. Recall that 
first-order inclusion logic was shown in \cite{inclusion_logic_GH} to be not expressively complete either, since some of the union closed existential second-order team properties cannot be defined in the logic.
While the union closed  fragment of existential second-order logic was already characterized in \cite{HoelzelWilke20} by using an involved fragment of inclusion-exclusion logic, it is reasonable to ask 
whether it is possible to find a simpler expressively complete union closed first-order team-based logic, by extending first-order logic with certain more general inclusion atoms, as done in the present paper on the propositional level. 

We have introduced sound and complete natural deduction systems  for  \PU and \PInc as well as \PIncz. 
How to axiomatize the logic \PAm is left as future work. 
The completeness proofs for the systems of \PU and \PInc makes heavy use of the disjunctive normal form of the two logics. Since the normal form of \PAm (\Cref{nf-pinc-picnst}(\ref{nf-pinc-picnst-1})) is substantially more involved, in order to obtain an elegant proof system for \PAm, one may need to take a different approach, or at least to formulate a simpler normal form for the logic.
%
Introducing (cut-free) sequent calculi for all these union closed logics and investigating their proof-theoretic properties are natural further directions. For propositional downwards closed logics, some first steps along this line were taken in \cite{FrittellaGrecoandPalmigianoYang16,IemhoffYang15}.



We have also analyzed the locality property in union closed team logics. We stressed that this simple property should not be taken for granted in the context of team semantics. We gave examples to illustrate that the union closed team logics considered in this paper under strict semantics actually lack the locality property. We have also briefly discussed that the locality property can actually be decomposed into the upwards and downwards locality. While this distinction between upwards and downwards locality did not lead to  new result in this paper,
it is our hope that this distinction can inspire further research on locality.
We have also discussed the connection between locality and interpolation. We reformulated the proof of the result in \cite{DAgostino19} that given  locality, all expressively complete team-based propositional logics (in some forgetful class) enjoy uniform interpolation (\Cref{uni_interpo}). We highlighted the subtle but crucial role that the locality property play in the argument for this result. This then naturally raises the question whether locality is actually a presupposition for interpolation. The example we gave in \Cref{fail_interp_exam} is at least consistent with this idea. A thorough investigation into the connection between locality and interpolation is left for future work.


We end by mentioning two other further directions. First is to find  applications of union closure team logics in other fields. Propositional downwards closed team logics have natural interpretations in inquisitive semantics (see e.g., \cite{Ciardelli2015}). Developing similar connections for union closed team logics in natural language and other contexts would be an interesting further direction. Along this line, recent work by Aloni \cite{Aloni2021} used a union closed team-based modal logic with the atom {\footnotesize$\mathsf{NE}$}  to model {\em free-choice} inferences in natural language, where {\footnotesize$\mathsf{NE}$} states that the team in question is nonempty (and thus the resulting logic, studied in \cite{Anttila2021}, does not satisfy the empty team property).  
Another interesting direction is to consider team-based (propositional) logics with other closure properties. The team-based logics considered in the literature are usually conservative extensions of classical logic. The characteristic property of classical propositional formulas is the flatness property  (\Cref{flat_char}), which is equivalent to the combination of the empty team property, the union closure property and the downwards closure property. In this respect,  
the union closure and downwards closure property are  natural closure properties for team-based logics. In contrast, for instance, the upwards closure property is  not very natural, because, as we pointed out already, classical formulas (e.g., already the propositional variable $p$) are not upwards closed. Nevertheless, there may well be other meaningful ways to decompose the flatness property. For an obvious example, the flatness property is stated as a property of two directions, each of which  corresponds to a closure property that has not yet been considered in the literature. The closure properties obtained from this and possibly other decompositions could give rise to other interesting logics.








\section*{Acknowledgments}

The author would like to thank Aleksi Anttila,  Fausto Barbero, Pietro Galliani, Rosalie Iemhoff, Juha Kontinen, Lauri Hella, and Jouko V\"{a}\"{a}n\"{a}nen for interesting discussions related to the topic of this paper. The author is also grateful to an anonymous referee for valuable comments concerning the presentation of the results in the paper. 

This research was supported by grants 330525 and 308712 of Academy of Finland, and  Research Funds of  University of Helsinki.

\vspace{-8pt}

\section*{References}


\end{document}